\documentclass[10pt,namelimits,sumlimits]{amsart}
\usepackage{appendix}
\usepackage{amssymb,amsmath, mathrsfs}
\usepackage[mathscr]{eucal}
\usepackage{color}
\usepackage{enumitem}
\usepackage[utf8]{inputenc}
\usepackage[leftcaption]{sidecap}
\usepackage{tikz}
 \usepackage{tikz-cd}
 \usepackage{stmaryrd}
\usetikzlibrary{matrix,arrows,decorations.pathmorphing}
\usetikzlibrary{positioning}
\usetikzlibrary{matrix,arrows,decorations.pathmorphing, shapes.geometric}

\usepackage{psfrag}

\usepackage{epstopdf}
\usepackage{graphicx}
\usepackage{verbatim}
\usepackage{amsmath,amsthm}
\usepackage{graphics}
\usepackage{color}
\usepackage{epsfig}
\usepackage{amssymb,amsmath}
\usepackage[mathscr]{eucal}
\usepackage{ upgreek }
\usepackage{tensor}
\usepackage{physics}

\newtheorem{theorem}[equation]{Theorem}
\newtheorem{conjecture}[equation]{Conjecture}
\newtheorem{lemma}[equation]{Lemma}
\newtheorem{prop}[equation]{Proposition}

\newtheorem{corollary}[equation]{Corollary}

\newtheorem{definition}[equation]{Definition}

\theoremstyle{remark}
\newtheorem{remark}[equation]{Remark}
\newtheorem{notation}[equation]{Notation}
\newtheorem{convention}[equation]{Convention}
\newtheorem{assumption}[equation]{Assumption}

\numberwithin{equation}{section}

\DeclareMathOperator{\sgn}{sgn}
\newcommand{\Sim}{\displaystyle\operatornamewithlimits{\sim}}

\newcommand{\utilde}{{\widetilde{u}}}
\newcommand{\C}{\mathbb{C}}
\newcommand{\R}{\mathbb{R}}

\newcommand{\B}{\mathbb{B}}
\newcommand{\N}{\mathbb{N}}
\newcommand{\D}{\mathbb{D}}
\newcommand{\cat}{\mathbb{K}}
\newcommand{\K}{\mathbb{K}}
\newcommand{\Rcal}{\mathcal{R}}
\newcommand{\Jcal}{\mathcal{J}}
\newcommand{\Lcal}{\mathcal{L}}

\newcommand{\Lcalhatcat}{\widehat{\Lcal}_{\cat}}

\newcommand{\group}{{\mathscr{G}  }}

\newcommand{\grouptheta}{\group^{\varpi}}

\newcommand{\supp}{\operatorname{supp}}

\newcommand{\groupdisk}{\group}

\newcommand{\aunder}{\underline{a}}

\newcommand{\Bcal}{\mathcal{B}}
\newcommand{\sym}{sym}

\newcommand{\Thetacyl}{Y_{\cyl}}

\newcommand{\psicut}{{\psi_{\mathrm{cut}}}}
\newcommand{\Psibold}{{\boldsymbol{\Psi}}}

\newcommand{\Sph}{\mathbb{S}}

\newcommand{\what}{\hat{w}}

\newcommand{\id}{\operatorname{Id}}

\newcommand{\gaux}{g_A}

\newcommand{\gtilde}{\widetilde{g}}

\newcommand{\Etilde}{\widetilde{E}}

\newcommand{\Mcal}{\mathcal{M}}

\newcommand{\Rcap}{\mathsf{R}}
\newcommand{\Rcapunder}{\underline{\Rcap}}
\newcommand{\phicat}{{\varphi_{\mathrm{cat}}}}
\newcommand{\phigraph}{{\overline{\varphi}}}
\newcommand{\arccosh}{\operatorname{arccosh}}
\newcommand{\dbold}{{\mathbf{d}}}
\newcommand{\varphigl}{\varphi^{gl}}
\newcommand{\cyl}{\mathrm{Cyl}}
\newcommand{\cyla}{\cyl_{[0, a]}}
\newcommand{\vunder}{\underline{v}}
\newcommand{\Sigmacir}{\mathring{M}}
\newcommand{\Sigmacircat}{\mathring{M}_{\K}}
\newcommand{\Sigmacird}{\Sigmacir_{\D}}
\newcommand{\Mcir}{\mathring{M}}

\newcommand{\kappahat}{\widehat{\kappa}}
\newcommand{\kappacir}{\mathring{\kappa}}
\newcommand{\Xcir}{\mathring{X}}
\newcommand{\nucir}{\mathring{\nu}}
\newcommand{\gcir}{\mathring{g}}
\newcommand{\Thetacir}{\mathring{\Theta}}

\newcommand{\ucir}{\mathring{u}}
\newcommand{\Ucir}{\mathring{U}}
\newcommand{\scir}{\mathring{s}}
\newcommand{\sigmacir}{\mathring{\sigma}}
\newcommand{\Mcird}{\Mcir_{\D}}
\newcommand{\rhocir}{\mathring{\rho}}
\newcommand{\Acir}{\mathring{A}}
\newcommand{\Hcir}{\mathring{H}}

\newcommand{\Dp}{\D^+}

\newcommand{\Qtilde}{\widetilde{Q}}
\newcommand{\Psicat}{\Psi_{\cat}}
\newcommand{\Psidisk}{\Psi_{\D}}
\newcommand{\Sigmacat}{M_{\cat}}
\newcommand{\Sigmadisk}{M_{\D}}

\newcommand{\Rcalhat}{\widehat{\Rcal}}
\newcommand{\betahat}{\widehat{\beta}}
\newcommand{\gammahat}{\widehat{\gamma}}
\newcommand{\picir}{\mathring{\pi}}
\newcommand{\graph}{\operatorname{Graph}}

\newcommand{\Sigmad}{M_{\D}}
\newcommand{\tauunder}{\underline{\tau}}

\newcommand{\vhat}{\hat{v}}
\newcommand{\rr}{\mathrm{r}}

\newcommand{\dom}{\mathrm{dom}}
\newcommand{\inj}{{\mathrm{inj}}}

\newcommand{\DDD}{\mathsf{D}}
\newcommand{\Utilde}{\widetilde{U}}  
\newcommand{\catbr}{{\check{K}}}
\newcommand{\phicataux}{\underline{\varphi}_{\mathrm{cat}}}

\newcommand{\skernel}{\mathscr{K}}
\newcommand{\skernelv}{\widehat{\mathscr{K}}}
\newcommand{\Fzeta}{\mathcal{F}_\zeta}
\newcommand{\phibreve}{\breve{\phi}}
\newcommand{\zetabreve}{\breve{\zeta}}
\newcommand{\Mbreve}{\breve{M}}
\newcommand{\Mhat}{\widehat{M}}

\newcommand{\tunder}{\underline{t}}
\newcommand{\Fzetacir}{\mathring{\mathcal{F}}_{\zeta}}
\newcommand{\Fcalcir}{\mathring{\mathcal{F}}}
\newcommand{\Pcir}{\mathring{P}}
\newcommand{\Ttilde}{\widetilde{T}}
\newcommand{\Pprime}{P_{\varpi}'}

\begin{document}

\title[Immersed Free Boundary Annuli]{Free Boundary Minimal Annuli Immersed in the Unit 3-Ball}

\author[N.~Kapouleas]{Nikolaos~Kapouleas}
\author[P.~McGrath]{Peter~McGrath}

\address{Department of Mathematics, Brown University, Providence,
RI 02912}  \email{nicolaos\_kapouleas@brown.edu}

\address{Department of Mathematics, North Carolina State University, Raleigh NC 27695} 
\email{pjmcgrat@ncsu.edu}

\date{\today}

\begin{abstract}
Using the linearized doubling methodology we carry out a PDE gluing construction of a 
discrete family of non-rotational properly immersed free boundary minimal annuli in the Euclidean unit 3-ball.   
The surfaces we construct resemble equatorial disks joined by half-catenoidal bridges at the boundary.
\end{abstract}
\maketitle

\section{Introduction}
\label{S:intro}
\nopagebreak

\subsection*{The general framework}
\label{s:framework} 
$\phantom{ab}$
\nopagebreak

The simpest examples of properly embedded free boundary minimal (FBM) surfaces in the Euclidean three-ball $\B^3$ are the equatorial disk and the critical catenoid. 
Nitsche in 1985 proved that the equatorial disk is the only FBM immersed disk in $\B^3$ by using Hopf-differential methods \cite{Nitsche}. 
More recently Fraser-Schoen proved that this holds in $\B^n$ for any $n>3$ \cite{fraser-disk}. 
Fraser-Li formalized the following conjecture which is analogous to the Lawson conjecture for embedded minimal tori in the round three-sphere. 

\begin{conjecture}[Fraser and Li, \cite{FraserLi}]
\label{Ccc}
Up to congruence, the critical catenoid is the only properly embedded free boundary minimal annulus in $\B^3$.
\end{conjecture}

Although the Lawson conjecture was proven in a celebrated result of Brendle \cite{brendle}, 
only partial results \cite{FSinvent, McGrath, KusnerMcGrath, Seo} towards Conjecture \ref{Ccc} are known.  
In particular, Kusner and McGrath \cite{KusnerMcGrath} have verified Conjecture \ref{Ccc} under the additional hypothesis that the embedded annulus is antipodally symmetric.
There have been claims that the critical catenoid and its coverings are the only immersed free boundary minimal annuli in $\B^3$ \cite{Nitsche, Nadirashvili, CCLiu, Aram}.
In this article we construct FBM immersed annuli in $\B^3$ contradicting such claims. 
Before we proceed, we mention that the FBM surfaces constructed in this article resemble 
surfaces produced recently using a beautiful classical construction (with totally different methodology from our methodology) by Fern\'andez-Hauswirth-Mira \cite{Fernandez}. 
There is some overlap between the broad ideas of the two constructions although they were developed independently and before \cite{Fernandez} appeared. 

\subsection*{Brief discussion of the results}
\label{s:results} 
$\phantom{ab}$
\nopagebreak

Our constructions are partially motivated by our expectation that there is a one-parameter continuous family of periodic immersed FBM strips,  
analogous in some ways to the one-parameter continuous family of periodic CMC surfaces in Euclidean three-space discovered by Delaunay in 1841. 
Instead of the translational period the Delaunay surfaces have, the period of these FBM strips would be a rotation $\Rcap^{\pi+2\varpi}_\ell$ about an axis $\ell$ of angle $\pi+2\varpi$;  
and instead of the rotational invariance of the Delauny surfaces, they would only have reflectional symmetry with respect to a plane $\ell^\perp$ perpendicular to $\ell$ 
(see \ref{Eell}, \ref{drot}, \ref{dLD} and \ref{Tmain}). 
When $\varpi$ is a rational multiple of $\pi$ the immersion would factor through an annulus providing the desired examples. 

The embedded Delaunay surfaces (also called \emph{unduloids}) interpolate between a cylinder and a ``necklace'' of spheres.  
In analogy we expect that the FBM strips interpolate between the universal cover of the critical catenoid and a ``necklace'' of equatorial disks. 
We were successful in constructing a continuous family of immersed FBM strips as small perturbations of the universal cover of the critical catenoid 
but establishing that $\varpi/\pi$ is not an irrational constant seems challenging. 
Concentrating at the other end of the moduli space we construct periodic FBM strips resembling a sequence of equatorial disks joined by small half-catenoidal bridges at the boundary.  
$\varpi$ can then be prescribed and so we obtain the desired immersed annuli as described below. 

The main results of this article are presented in Theorem \ref{Tmain} and Corollary \ref{Cmain}. 
In Theorem \ref{Tmain} we prove the existence of the fundamental domain $\Mbreve$ of our FBM strips or annuli.  
The fundamental domain $\Mbreve$ is a smooth disc with corners and in Corollary \ref{Cmain} we prove that $\Mbreve$ generates the whole strip or annulus 
by successive reflections with respect to planes through $\ell$. 

The fundamental domain $\Mbreve$ resembles an equatorial half-disk with half of a half-catenoidal (at the boundary $\partial\B^3=\Sph^2$) bridge attached. 
The part resembling an equatorial-half disk is a small perturbation of an actual equatorial half-disk $\D^+$ tilted by an angle $\varpi$ about $\ell$  
($\D^+$ is separated from its other half by $\ell$ and $\ell\cap\B^3 \subset \partial \D^+$).   
We call $P$ the equatorial plane of $\D^+$ and $\Pprime := \Rcap^{\frac{\pi}{2}+\varpi}_\ell P$ 
the orthogonal plane to $P$ through $\ell$ rotated by an angle $\varpi$ about $\ell$ (see \ref{dPprime}). 
The half-catenoidal bridge resembles the intersection with $\B^3$ of a small truncated catenoid positioned with its waist on $P$ and centered at 
$p$ where $\{p\}=\Sph^2\cap\ell^\perp\cap\partial\D^+$. 

The boundary $\partial\Mbreve$ is the union of the following four smooth arcs. 
First, $\Mbreve\cap\Pprime$ where $\Mbreve$ and $\Pprime$ meet orthogonally.
Second, $\Mbreve\cap P$, which is the waist of the half-catenoidal bridge, where $\Mbreve$ and $P$ meet orthogonally. 
Finally, the two connected components of $\Mbreve\cap\Sph^2$, which are exchanged by the reflection $\Rcapunder_{\ell^\perp}$ with respect to $\ell^\perp$ (see \ref{dgroupd}), 
where $\Mbreve$ and $\Sph^2$ meet orthogonally.   
$\Mbreve$ is therefore a smooth disc with four corners at the boundary of angle $\pi/2$ each. 
Reflections with respect to either $P$ or $\Pprime$ extend then $\Mbreve$ 
smoothly to a domain which can give the whole strip or annulus by applying successively the rotation $\Rcap^{\pi+2\varpi}_\ell$.  

\subsection*{Outline of strategy and main ideas}
\label{s:strategy} 
$\phantom{ab}$
\nopagebreak

The construction follows the linearized doubling (LD) methodology introduced in \cite{kap} and further developed in \cite{kap-mcg-cpam, KapMcGLDG}, appropriately modified.  
Under this methodology one starts with a base surface $\Sigma$ which is minimal and on which LD solutions are constructed. 
LD solutions are singular solutions of the Jacobi equation with logarithmic singularities at finitely many singular points. 
By attaching then small catenoidal bridges to the graphs of slighlty modified LD solutions at the vicinity of the singular points one constructs initial surfaces 
which are approximately minimal. 
Under suitable assumptions one of the initial surfaces can be corrected to minimality providing a minimal doubling of $\Sigma$. 

The LD approach has been already modified in \cite[Section 12]{KapMcGLDG} along the lines of \cite{KapWiygul} 
to apply to the construction of FBMS doublings of the critical catenoid.  
In the current construction however we face two new difficulties. 
First, 
there is no suitable minimal base surface for the construction of $\Mbreve$.    
Second, the half-catenoidal bridges are at the boundary. 
We resolve the first difficulty by allowing base surfaces which are not minimal. 
More precisely we use a base surface $\Sigma$ which interpolates between $\D^+$ in the vicinity of $p$ 
and $\D^+$ tilted by an angle $\varpi$ in the vicinity of $\ell$ (see \ref{dbase} and \ref{Lbprop}). 

The second dificulty is resolved by allowing singular points of the LD solutions at the boundary (see \ref{dLD}) and modifying the construction accordingly. 
In particular because of the boundary conditions we do not need to consider the first harmonics on the parallel half-circles of the half-catenoidal bridges vanishing in the vicinity of the boundary.
Because of the symmetry with respect to $\ell^\perp$ we do not need to consider the orthogonal first harmonics either. 
This results in a simplified mismatch (see \ref{Lmismatch}) with a constant term only and no linear terms. 
Similarly there is no need to tilt the half-catenoidal bridges simplifying the construction of the pre-initial surfaces (see \ref{dpreinit}).   
On the other hand a new term $\vhat$ has to be added to correct for the mean curvature of base surface. 

We summarize now the various steps of the construction. 
We first fix $\varpi$ and construct a family of LD solutions on $\D^+$ (given in closed form) parametrized by a single parameter $\zeta$ (see \ref{dLDf}, \ref{dLD} and \ref{Lldexist}). 
The LD solutions on $\D^+$ are then transplanted to the base surface $\Sigma$ and appropriately modified using $\vhat$ as discussed above and $\vunder$ in an obstruction space $\skernelv$ 
(see \ref{dobs}) to give a family of functions (similar to LD solutions) $\phigraph $ on $\Sigma$.  
The strength of the singularity of  $\phigraph$ is comparable to the period controling parameter $\varpi$ (see \ref{Etau}).

The graphs of the functions $\phigraph$ define the pre-initial surfaces $\Sigmacir : = \Sigmacir[\zeta]$ away from the singular point $p$ (see \ref{dpreinit}); 
note that following \cite{KapWiygul} the graphs are taken with respect to an auxiliary metric as in \ref{dgaux} which ensures good behavior at the boundary. 
In the vicinity of $p$ the pre-initial sutfaces contain an appropriate half of a half-catenoidal bridge constructed as in \cite{KapWiygul} 
and glued to the previous part in a fashion similar to \cite[Definition 3.17]{KapMcGLDG}. 
The pre-initial surfaces $\Sigmacir$ are estimated in \ref{Lmcirest};  
they are not designed to intersect $\Sph^2$ orthogonally and this is corrected in the definition of the initial surfaces (see \ref{dinit}) 
by using a simple device which converts the boundary error to mean curvature in the interior as in \cite{KapWiygul}.  
Finally one of the pre-initial surfaces is corrected by a fixed point theorem argument as usual to provide the fundamental domain $\Mbreve$ in \ref{Tmain}. 

\subsection*{Organization of the presentation}
$\phantom{ab}$
\nopagebreak

Besides the Introduction the article has seven more Sections and one Appendix. 
In Section 2, we review definitions and notation related to the objects we are interested in.
In Section 3, we study LD solutions on the half-disk $\Dp$: after defining them in \ref{Lmismatch}, studying the mismatch in \ref{Lmismatch}, 
and establishing appropriate existence and uniqueness properties in \ref{Lldexist}, in \ref{dLDf} we parametrize the family of LD solutions used to construct the pre-initial surfaces.  
Section 4 is concerned with the base surfaces: in \ref{dbase} we define the base surface $\Sigma = \Sigma[\varpi]$ by bending the half-disk $\Dp$ in the vicinity of its diameter, 
and in \ref{Lbase}, we estimate geometric quantities on $\Sigma[\varpi]$.  
In Sections \ref{Spreinit} and \ref{Sinit} we define and estimate geometric quantities on the pre-initial and initial surfaces, 
and in Section \ref{S:linearized} we solve the linearized equation on the initial surfaces.  
In the last Section 8 of this article, in \ref{Tmain} and in \ref{Cmain} we state and prove the main results. 
Finally in the Appendix we collect some needed facts about graphs in the auxiliary metric.



\section{Notation and Conventions}
\label{S:notation}
\nopagebreak

\subsection*{General notation and conventions}
\label{sub:not}
\nopagebreak

\begin{notation}
\label{NT}
For $(N, g)$ is a Riemannian manifold, $\Sigma \subset N$ a two-sided hypersurface equipped with a (smooth) unit normal $\nu$, and $\Omega \subset \Sigma$, we introduce the following notation, where any of $N, g, \Sigma$, or $\Omega$ may be omitted when clear from context. 
\begin{enumerate}[label={(\roman*)}]
\item For $A\subset N$ we write $\dbold^{N, g}_A$ for the distance function from $A$ with respect to $g$ 
and we define the \emph{neighborhood of $A$ of radius $\delta>0$} by
$D^{N, g}_A(\delta):=\left \{p\in N:\dbold^{N, g}_A(p)<\delta\right\}. $  
If $A$ is finite we may just enumerate its points in both cases, for example if $A=\{q\}$ we write $\dbold_q(p)$. 
\item We denote by $\exp^{N, g}$ the exponential map, 
by $\dom(\exp^{N,g}) \subset TN$ its maximal domain, and by $\inj^{N, g}$ the injectivity radius of $(N,g)$.  Similarly by $\exp_p^{N, g}$, $\dom(\exp_p^{N,g})$ and $\inj_p^{N, g}$ the same at $p\in N$.
\item 
\label{Dpertimm}
Given also a vector field $V$ defined along $\Omega$ satisfying $V_p \in \dom(\exp^{N,g})$ for each $p \in \Omega$, 
we define 
\begin{align*}
X_{\Omega, V}^{N, g} : \Sigma \rightarrow N
\quad \text{ by } \quad
X_{\Omega, V}^{N, g} = \exp^{N, g} \circ V \circ I_{\Omega}^N,
\end{align*}
where $I_\Omega^N$ denotes the inclusion map of $\Omega$ in $N$. 
\item With $V$ and $X^{N,g}_{\Omega, V}$ as above, we use the notation
	\begin{align*}
	\graph^{N,g}_{\Omega}( V) : = X^{N,g}_{\Omega, V}(\Omega).
	\end{align*}
\end{enumerate}
\end{notation}

Our arguments require extensive use of cut-off functions and the following will 
be helpful. 
\begin{definition}
\label{DPsi} 
We fix a smooth function $\Psi:\R\to[0,1]$ with the following properties:
\begin{enumerate}[label=\emph{(\roman*)}]
\item $\Psi$ is nondecreasing.

\item $\Psi\equiv1$ on $[1,\infty)$ and $\Psi\equiv0$ on $(-\infty,-1]$.

\item $\Psi-\frac12$ is an odd function.
\end{enumerate}
\end{definition}

Given $a,b\in \R$ with $a\ne b$,
we define smooth functions
$\psicut[a,b]:\R\to[0,1]$
by
\begin{equation}
\label{Epsiab}
\psicut[a,b]:=\Psi\circ L_{a,b},
\end{equation}
where $L_{a,b}:\R\to\R$ is the linear function defined by the requirements $L_{a,b}(a)=-3$ and $L_{a, b}(b)=3$.

Clearly then $\psicut[a,b]$ has the following properties:
\begin{enumerate}[label={(\roman*)}]
\item $\psicut[a,b]$ is weakly monotone.

\item 
$\psicut[a,b]=1$ on a neighborhood of $b$ and 
$\psicut[a,b]=0$ on a neighborhood of $a$.

\item $\psicut[a,b]+\psicut[b,a]=1$ on $\R$.
\end{enumerate}

Suppose now we have two sections $f_0,f_1$ of some vector bundle over some domain $\Omega$.
(A special case is when the vector bundle is trivial and $f_0,f_1$ real-valued functions).
Suppose we also have some real-valued function $d$ defined on $\Omega$.
We define a new section 
\begin{equation}
\label{EPsibold}
\Psibold\left [a,b;d \, \right](f_0,f_1):=
\psicut[a,b ]\circ d \, f_1
+
\psicut[b,a]\circ  d \, f_0.
\end{equation}
Note that
$\Psibold[a,b;d\, ](f_0,f_1)$
is then a section which depends linearly on the pair $(f_0,f_1)$
and transits from $f_0$
on $\Omega_a$ to $f_1$ on $\Omega_b$,
where $\Omega_a$ and $\Omega_b$ are subsets of $\Omega$ which contain
$d^{-1}(a)$ and $d^{-1}(b)$ respectively,
and are defined by
$$
\Omega_a=d^{-1}\left((-\infty,a+\frac13(b-a))\right),
\qquad
\Omega_b=d^{-1}\left((b-\frac13(b-a),\infty)\right),
$$
when $a<b$, and 
$$
\Omega_a=d^{-1}\left((a-\frac13(a-b),\infty)\right),
\qquad
\Omega_b=d^{-1}\left((-\infty,b+\frac13(a-b))\right),
$$
when $b<a$.
Clearly if $f_0,f_1,$ and $d$ are smooth then
$\Psibold[a,b;d\, ](f_0,f_1)$
is also smooth.

In comparing equivalent norms or other quantities we will find the following notation useful. 
\begin{definition}
\label{Dsimc}
We write $a\Sim_c b$ to mean that 
$a,b\in\R$ are nonzero of the same sign, 
$c\in(1,\infty)$, 
and $\frac1c\le \frac ab \le c$. 
\end{definition}

We use the standard notation $\left\|u: C^{k,\beta}(\,\Omega,g\,)\,\right\|$ 
to denote the standard $C^{k,\beta}$-norm of a function or more generally
tensor field $u$ on a domain $\Omega$ equipped with a Riemannian metric $g$.
Actually the definition is completely standard only when $\beta=0$
because then we just use the covariant derivatives and take a supremum
norm when they are measured by $g$.
When $\beta\ne0$ we have to use parallel transport along geodesic segments 
connecting any two points of small enough distance
and this may be a complication if small enough geodesic balls are not convex.
In this paper we take care to avoid situations where such a complication
may arise and so we will not discuss this issue further.

We adopt the following notation from \cite{kap} for weighted H\"{o}lder norms.  
\begin{definition}
\label{dwHolder}
Assuming that $\Omega$ is a domain inside a manifold,
$g$ is a Riemannian metric on the manifold, 
$k\in \N_0$, 
$\beta\in[0,1)$, $u\in C^{k,\beta}_{\mathrm{loc}}(\Omega)$ 
or more generally $u$ is a $C^{k,\beta}_{\mathrm{loc}}$ tensor field 
(section of a vector bundle) on $\Omega$, 
$\rho,f:\Omega\to(0,\infty)$ are given functions, 
and that the injectivity radius in the manifold around each point $x$ in the metric $\rho^{-2}(x)\,g$
is at least $1/10$,
we define
$$
\left\|u: C^{k,\beta} ( \Omega,\rho,g,f)\right\|:=
\sup_{x\in\Omega}\frac{\,\left\|u:C^{k,\beta}(\Omega\cap B_x, \rho^{-2}(x)\,g)\right\|\,}{f(x) },
$$
where $B_x$ is a geodesic ball centered at $x$ and of radius $1/100$ in the metric $\rho^{-2}(x)\,g$.
For simplicity we may omit any of $\beta$, $\rho$, or $f$, 
when $\beta=0$, $\rho\equiv1$, or $f\equiv1$, respectively.
\end{definition}

$f$ can be thought of as a ``weight'' function because $f(x)$ controls the size of $u$ in the vicinity of
the point $x$.
$\rho$ can be thought of as a function which determines the ``natural scale'' $\rho(x)$
at the vicinity of each point $x$.
Note that if $u$ scales nontrivially we can modify appropriately $f$ by multiplying by the appropriate 
power of $\rho$.  Observe from the definition the following multiplicative property: 
\begin{equation}
\label{Enormmult}
\left\| \, u_1 u_2 \, : C^{k,\beta}(\Omega,\rho,g,\, f_1 f_2 \, )\right\|
\le
C(k)\, 
\left\| \, u_1 \, : C^{k,\beta}(\Omega,\rho,g,\, f_1 \, )\right\|
\,\,
\left\| \, u_2 \, : C^{k,\beta}(\Omega,\rho,g,\, f_2 \, )\right\|.
\end{equation}

\subsection*{Free Boundary Minimal Surfaces and background elementary geometry}
\label{Sfb}
\nopagebreak

\begin{definition}
\label{dfbms}
Let $(N, g)$ be a Riemannian manifold and $\Omega \subset N$ be a domain with smooth boundary.  
A smooth, properly immersed (in the sense that intersections with compact subsets of $\Omega$ are compact) 
submanifold $\Sigma^k \subset \Omega$ is a \emph{free boundary minimal submanifold} if its mean curvature vanishes, $\partial \Sigma \subset \partial \Omega$, and $\Sigma$ meets $\partial \Omega$ orthogonally along $\partial \Sigma$.
\end{definition}

In this article, we will be interested in free boundary minimal surfaces in the Euclidean unit ball $\B^3$.  The following notation is convenient for dealing with surfaces whose boundaries only partially lie on $\Sph^2 : = \partial \B^3$. 

\begin{notation}
\label{Nsphbd}
Given a surface $S \subset \B^3$ with boundary, we define $\partial_{\Sph^2} S: =\partial S \cap \partial \B^3$.
\end{notation}

\begin{definition}[Jacobi operators]
Let $S \subset \B^3$ be a smooth surface, with unit outward conormal field $\eta$ along $\partial \Sigma$.  We define the Jacobi operator $\Lcal_S$ and the boundary Jacobi operator $\Bcal_{S}$ by
\begin{align*}
\Lcal_S  : = \Delta + |A|^2
\quad 
\text{and}
\quad
\Bcal_S  : = -\partial_{\eta} + 1.
\end{align*}
\end{definition}

By standard calculations, the linearized equation for free boundary minimal surfaces in $\B^3$ 
at a free boundary minimal surface $S$ in $\B^3$ defined as in \ref{dfbms}, 
is given (see for example \cite[2.25, (2.31) and (2.41)]{kapli}) by the boundary value problem
\begin{equation} 
\left\{ \,\, 
\label{Ejacfb}
\begin{aligned}
\Lcal_S u \, &= \,\, 0 \quad \text{on} \quad S, \\
\Bcal_S u  \, &= \,\, 0 \quad \text{on} \quad \partial S. 
\end{aligned}
\right. 
\end{equation}

Throughout this article, $\R^3$ will denote the Euclidean 3-space, equipped with its standard orientation.  Given a vector subspace $V$ of $\R^3$, we denote by $V^\perp$ its orthogonal complement in $\R^3$, and we define the reflection in $\R^3$ with respect to $V$, $\Rcapunder_V : \R^3 \rightarrow \R^3$, by
\begin{align*}
\Rcapunder_V : = \Pi_V - \Pi_{V^\perp},
\end{align*}
 where $\Pi_V$ is the orthogonal projection of $\R^3$ onto $V$.  
 
 \begin{definition}[Rotations $\Rcap_V^{\phi}$]
\label{drot}
Given a one-dimensional subspace $V \subset \R^3$, $\phi \in \R$, and an orientation chosen on the orthogonal complement $V^\perp$, the rotation about $V$ by angle $\phi$ is defined to be the element $\Rcap_V^{\phi}$ of $SO(3)$ preserving $V$ pointwise and rotating $V^\perp$ along itself by angle $\phi$ (in accordance with its chosen orientation).
\end{definition}

Fix the two-dimensional linear subspace $P = \{(x,y,0) : x, y \in \R\}$ of $\R^3$ and the point $p = (1, 0, 0)$.  
We define the unit disk $\D$, unit half-disk $\Dp$, and the line $\ell$ by
\begin{align}
\label{Eell} 
 \D : = \B^3 \cap P
 \quad \text{and} \quad 
 \D^+ : = \{ q \in \D : q\cdot p \geq 0\}, 
 \quad
 \ell = (\mathrm{span}(p))^\perp \cap P.
 \end{align}
We denote by $(\rr, \theta)$ the standard polar coordinates on $P \setminus \{ (x, 0, 0): x \leq 0\}$, 
by $\eta$ the outward pointing conormal field to $\D$ along $\partial \D$, 
and by $\nu$ the upward pointing unit normal field to $P$.  
Finally, it will be useful (see for example \ref{Lldexist}(ii)) to identify $P$ with $\C$ by identifying each $(x,y,0) \in P$ with $x+iy \in \C$.  

\begin{definition}
\label{dv}
Let $v : P \rightarrow \R$ be the linear map defined by $v(\cdot) : = \langle p, \cdot \rangle$. 
\end{definition}

\begin{assumption}
\label{contheta}
We fix now some $\varpi>0$ which we will assume as small in absolute terms as needed. 
\end{assumption}

\begin{definition}[Symmetry groups and the plane $\Pprime$] 
\label{dgroupd}
\label{dPprime} 
Let $\groupdisk$ be the group of order two generated by the reflection $\Rcapunder_{\ell^\perp}$.
Moreover given $\varpi>0$ as in \ref{contheta}, 
let $\grouptheta$ be the subgroup of $O(3)$ generated by $\Rcapunder_{\ell^\perp}, \Rcapunder_P$ and $\Rcapunder_{\Pprime}$, 
where $\Pprime : = \Rcap^{\frac{\pi}{2}+\varpi}_\ell P$. 
\end{definition}

\begin{notation}
\label{Nsymd}
If $X$ is a function space consisting of functions on a set $\Omega \subset \R^3$ and $\Omega$ is invariant under the action of $\groupdisk$, we use a subscript ``sym" to denote the subspace $X_{\sym} \subset X$ consisting of those functions $f\in X$ which are invariant under the action of $\groupdisk$.
\end{notation}

It will be useful to foliate $\Dp$ by curves which intersect $\partial_{\Sph^2} \D$ orthogonally.  This can be achieved by the following. 

\begin{lemma}
\label{Lconfo}
The map $f: \C \setminus \{-1\} \rightarrow \C\setminus \{-1\}$ defined by $f(z) = \frac{1-z}{1+z}$ satisfies the following properties. 
\begin{enumerate}[label = \emph{(\roman*)} ]
\item $f$ is a conformal diffeomorphism.
\item $f^{-1} = f$.
\item $f$ maps $p = 1$ to $0$ and maps $\partial_{\Sph^2} \Dp$ onto $\D \cap \ell$. 
\end{enumerate}
\end{lemma}
\begin{proof}
This follows from straightforward calculations which we omit. 
\end{proof}

\section{LD solutions}
\label{SLD}
\nopagebreak

\subsection*{LD solutions}
$\phantom{ab}$
\nopagebreak

\begin{definition}[LD solutions]
\label{dLD}
We call $\varphi$ a linearized doubling (LD) solution on $\Dp$ when there exists a number $\tau \neq 0$ and a number $\varpi \in \R$ such that the following hold.
\begin{enumerate}[label = \emph{(\roman*)} ]
\item $\varphi \in C^\infty_{\sym}(\Dp \setminus \{ p\})$.
\item $\Lcal_{\D} \varphi = 0$ on $\Dp \setminus \{ p\}$. 
\item $\Bcal_{\D} \varphi |_{\partial_{\Sph^2} \Dp} = 0$.
\item $\partial_\eta \varphi = - \varpi$ on $\partial \Dp \setminus \partial_{\Sph^2} \Dp$.  
\item $\varphi - \tau \log \dbold_p$ is bounded on $\Dp \setminus \{p\}$.  
\end{enumerate}
\end{definition}

\begin{lemma}[Existence and uniqueness for LD solutions]
\label{Lldexist}
For any $\tau \neq 0$ and any $\varpi \in \R$, there is a unique LD solution $\varphi = \varphi[\varpi, \tau]$ as in \ref{dLD}.  Moreover, the following hold. 
\begin{enumerate}[label = \emph{(\roman*)} ]
\item $\varphi = \tau G + \varpi v$, where $v$ is as in \ref{dv} and $G : = \varphi[0, 1]$. 
\item  $G(z) = 1+ \Re \big ( z \log \frac{1-z}{1+z} \big)$.
 \item $\| G - \log (\dbold_p/2)-1 : C^k( \Dp \setminus \{p\}, \dbold_p, g, \dbold_p |\log \dbold_p|) \|
 \leq C(k)$. 
\end{enumerate}
\end{lemma}
\begin{proof}
We first prove the uniqueness part.  Suppose $\varphi_1$ and $\varphi_2$ are LD solutions with the same $(\varpi, \tau)$ pair.
By standard removable singularity results, $\varphi: = \varphi_1 - \varphi_2$ extends to a harmonic function on $\Dp$.  Since $\partial_\eta \varphi = 0$ along $\partial \Dp \cap \ell$, then $\varphi$ extends by even reflection to a harmonic function on $\D$.  We call this extension $\varphi$ as well.  Consequently $\varphi$ is in the kernel of \eqref{Ejacfb} with $S = \D$.  This kernel is spanned by the coordinate functions, so the symmetries imply that $\varphi = 0$, concluding the proof of the uniqueness part. 

We now prove the existence part.  It clearly suffices to prove that $G$ as defined in item (ii) above is an LD solution with $(\varpi, \tau) = (0, 1)$.  For any $z \in \D \setminus \{\pm p\}$, easy calculations show that $G(z) = G(-z)$ and $G(z) = G(\overline{z})$.   It follows that $G$ satisfies the symmetries in \ref{dLD}(i) and the boundary condition in \ref{dLD}(iv).  The smoothness in \ref{dLD}(i) is obvious.  Next,  \ref{dLD}(ii) follows from the analyticity of $G$ on $\D \setminus \{\pm p\}$, and \ref{dLD}(v) is clear from inspection.  To verify \ref{dLD}(iii), note that the Cauchy-Riemann equations imply  $\partial_\eta ( \Re G(z) ) = \Re ( z G'(z) )$ on $\partial \D \setminus \{\pm p\}$, so 
\begin{align*}
\partial_\eta G (z) &= \Re \left( z \log \frac{1-z}{1+z} + \frac{2z^2}{z^2-1}\right)\\
&= \Re \left( z \log \frac{1-z}{1+z} \right) + \frac{z^2}{z^2-1} + \frac{\bar{z}^2}{\bar{z}^2-1}\\
&=  \Re \left( z \log \frac{1-z}{1+z} \right) +1 = G(z).
\end{align*}
This completes the proof of the existence part.  Finally, note that
\begin{align*}
z \log \frac{1-z}{1+z} = \log (1-z) + (z-1) \log (1-z) - z \log (1+z).
\end{align*}
Using that $\Re\left(\log (1-z) \right) = \log |1-z| = \log \dbold_p$ in combination with the preceding, the estimate (iii) follows. 
\end{proof}

\subsection*{Mismatch}
\phantom{ab}
\nopagebreak

\begin{lemma}[Mismatch of LD solutions]
\label{Lmismatch}
Given $\varphi$ as in \ref{dLD} with $\tau>0$, there is a number $\Mcal \varphi \in \R$ called the \emph{mismatch of $\varphi$}, defined uniquely by item (i) below, and moreover satisfying items (ii) and (iii).
\begin{enumerate}[label = \emph{(\roman*)} ]
\item $\varphi = \tau \log (2 \dbold_p/ \tau) + \Mcal \varphi + O( \dbold_p| \log \dbold_p|)$ near $p$.
\item $\Mcal \varphi = \tau \log \frac{e \tau}{4} + \varpi$. 
\item $\| \varphi - \tau \log (2 \dbold_p/\tau) - \Mcal \varphi : C^k(\Dp \setminus \{ p\}, \dbold_p, g, \dbold_p | \log \dbold_p|)\| \leq C(k)\tau.$
\end{enumerate}
\end{lemma}
\begin{proof}
By combining \ref{Lldexist}(iii) and \ref{dv}, we find that near $p$,
\begin{align*}
\varphi = \tau \log (\dbold_p /2) + \tau + \varpi + O\left(  \tau \dbold_p |\log \dbold_p| + |\varpi |\dbold_p\right).
\end{align*}
Items (i) and (ii) follow from this expansion, and item (iii) follows from the preceding and \ref{Lldexist}(iii).
\end{proof}

We are now ready to parametrize the family of LD solutions we will use. 
\begin{definition}[LD solutions {$\varphi\llbracket \zeta; \varpi\rrbracket$}]
\label{dLDf}
Given $\varpi$ as in \ref{contheta} and $\zeta \in [-\tauunder, \tauunder]$, where $\tauunder$ is defined below,  we define using \ref{Lldexist} the LD solution
\begin{equation}
\label{Etau}
\begin{aligned}
\varphi = \varphi\llbracket \zeta; \varpi \rrbracket = \varphi [ \varpi, \tau[\zeta, \varpi]],
\quad 
\text{where}
\quad
\tau[\zeta, \varpi] = e^{\frac{\zeta}{\varpi}} \tauunder
\quad 
\text{and}
\quad
\tauunder = e^{\frac{\log(\frac{4}{e} |\log \varpi|)}{\log \varpi}} \frac{\varpi}{|\log \varpi|} .
\quad
\end{aligned}
\end{equation}
\end{definition}

\begin{lemma}[Prescribed mismatch]
\label{Lpmis}
Let $\varphi = \varphi \llbracket \zeta; \varpi\rrbracket$ be as in \ref{dLDf}.  Then $| \Mcal \varphi+ \zeta |\leq \frac{1}{2}\tauunder.$
\end{lemma}
\begin{proof}
We first prove the estimate in the special case when $\zeta = 0$, that is, when $\tau = \tauunder$.  By \eqref{Etau}, we have
\begin{equation}
\label{Etau0m}
\begin{aligned}
\tauunder \log \tauunder 
&= 
e^{\frac{\log (\frac{4}{e}|\log \varpi|)}{\log \varpi}} \frac{\varpi}{|\log \varpi|} 
\left( \frac{\log(\frac{4}{e}|\log \varpi|)}{\log \varpi} + \log \varpi - \log |\log \varpi|\right)\\
&=  - \varpi e^{\frac{\log(\frac{4}{e} |\log \varpi|)}{\log \varpi}} - \tauunder \log |\log \varpi| + o(\tauunder)\\
&= -\varpi - \varpi \frac{\log(\frac{4}{e} |\log \varpi|)}{\log \varpi} - \tauunder \log |\log \varpi| + o(\tauunder)\\
&= -\varpi + \left(\frac{\varpi}{|\log \varpi|} - \tauunder \right) \log | \log \tau| - \frac{\varpi}{|\log \varpi|} \log \frac{e}{4}+  o(\tauunder)\\
&= -\varpi - \tauunder \log \frac{e}{4} + o(\tauunder),
\end{aligned}
\end{equation}
where we say that a function $f= f(\zeta, \varpi)$ is $o(\tauunder)$ and write $f(\zeta, \varpi) = o(\tauunder)$ if $\lim_{\tauunder \rightarrow 0} f(\zeta, \varpi)/\tauunder = 0$. 
By combining this with \ref{Lmismatch}, it follows that $\Mcal \varphi = o(\tauunder)$.  This proves the estimate in this case. 

We now consider the general case.  It follows from \eqref{Etau} that
\begin{align*}
\tau \log \tau = e^{\frac{\zeta}{\varpi}}\left(  \tauunder \log \tauunder + \frac{\zeta}{\varpi} \tauunder\right).
\end{align*}
Using from \eqref{Etau} that $|\zeta/ \varpi| \leq C/ |\log \varpi|$ and substituting from \eqref{Etau0m}, we find that
\begin{align*}
\tau \log \tau = - \varpi - \zeta- \tau \log \frac{e}{4} + o(\tauunder),
\end{align*}
or equivalently using \ref{Lmismatch}(ii) that $\Mcal \varphi + \zeta = o(\tauunder)$.
The conclusion now follows by combining this with \ref{Lmismatch} and taking $\varpi$ small enough.
\end{proof}

\begin{remark}
Although we will not need it, a more thorough accounting of the error terms in the proof of \ref{Lpmis} reveals that the estimate in \ref{Lpmis} can be improved to $|\Mcal \varphi + \zeta| \leq C \tauunder (\log |\log \varpi|)^2/|\log \varpi|$. 
\end{remark}

\subsection*{The obstruction space}
\phantom{ab}
\nopagebreak

\begin{definition}[The obstruction space]
\label{dobs}
Let $\skernelv \subset C^\infty_{\sym}(\Dp)$ be the subspace spanned by
\begin{align*}
\vunder = \Psibold[\delta, 2\delta; \dbold_0 \circ f](v, 0), 
\end{align*}
where $\delta := 1/100$ and $f$ was defined in \ref{Lconfo}.   We also denote by $\skernel \subset C^\infty_{\sym}(\Dp)$ the subspace spanned by $w : = \Lcal_{\D} \vunder$. 
\end{definition}

\begin{lemma}
\label{Lvmismatch}
The following hold.
\begin{enumerate}[label = \emph{(\roman*)} ]
\item $\| \vunder - 1 : C^k_{\sym} ( \Dp \setminus \{ p \}, \dbold_p, g, \dbold_p) \|
\leq C(k).$
\item $\Bcal_{\D} \vunder|_{\partial_{\Sph^2} \Dp} = 0$. 
\end{enumerate}
\end{lemma}
\begin{proof}
Item (i) follows from the definition of the norms and the fact that $\vunder -1$ vanishes at $p$.  We now prove item (ii).  Because $v$ is a coordinate function, it satisfies $\Bcal_{\D} v|_{\partial \D} =0$.  Since the level sets of $\dbold_0 \circ f$ meet $\partial_{\Sph^2} \D$ orthogonally by \ref{Lconfo}, it follows by combining the preceding with \ref{dobs} and \eqref{EPsibold} that (ii) holds. 
\end{proof}

\section{Base Surfaces}
\label{Sbase}

\subsection*{Definitions and basic properties}
$\phantom{ab}$
\nopagebreak

In order to construct surfaces with the appropriate symmetries, we first bend the half-disk $\Dp$.  We will do this in such a way that the bent disk meets $\partial \B^3$ orthogonally.

\begin{definition}[Base surfaces]
\label{dbase}
We define the base surface $\Sigma = \Sigma[\varpi]$ by $\Sigma = X_\Sigma(\Dp)$, where $X_\Sigma : \Dp \rightarrow \B^3$ is defined by
\begin{align}
\label{Ediffsig}
X_\Sigma(q) = \Rcap_\ell^{\Psibold(q)}(q)
\quad
\text{and}
\quad
\Psibold = \Psibold[\textstyle{\frac{9}{10}}, \frac{8}{10}; \dbold_0 \circ f](\varpi, 0).
\end{align}
Here we orient the plane $\ell^\perp$ (recall \ref{drot}) by requesting that $\{p, \nu\}$ is a positively oriented frame. 

We also denote by $A_\Sigma$ the second fundamental form of $\Sigma$, $H_\Sigma$ the mean curvature of $\Sigma$,  $\nu_\Sigma$ the upward pointing unit normal field, and $\eta_\Sigma$ the unit outward pointing conormal field to $\Sigma$ along $\partial \Sigma$.
\end{definition}

\begin{definition}
\label{dfconf}
Define $\Utilde_i \subset \D$, $U_i \subset \Dp$, and $U^\Sigma_i \subset \Sigma$, $i=1,2,3$, by
\begin{align*}
\Utilde_1 = D_0(1)\setminus D_0(9/10), 
\quad
\Utilde_2 = D_0(9/10)\setminus D_0(8/10),
\quad
\Utilde_3 = D_0(8/10),
\end{align*}
 $U_i: = f(\Utilde_i)$, and $U^\Sigma_i : = X_\Sigma(U_i)$, where $f$ is as in \ref{Lconfo}.  
\end{definition}

\begin{lemma}[Properties of the base surfaces]
\label{Lbprop}
For $\Sigma$ and $X_\Sigma$ as defined in \ref{dbase}, the following hold. 
\begin{enumerate}[label = \emph{(\roman*)} ]
\item $X_\Sigma$ is a smooth, $\groupdisk$-invariant embedding. 
\item $X_\Sigma|_{U_1} = \Rcap_{\ell}^{\varpi}|_{U_1} $.  In particular, $U^\Sigma_1 \subset \Rcap^{\varpi}_{\ell} P$ and $\Sigma \cup \Rcapunder_{\Pprime} \Sigma$ is a smooth disk. 
\item $X_{\Sigma}|_{U_3} = \id_{U_3}$.  In particular, $U^\Sigma_3 \subset \Dp$. 
\item  $\Sigma$ meets $\partial \B^3$ orthogonally along $\partial_{\Sph^2} \Sigma$. 
\end{enumerate}
\end{lemma}
\begin{proof}
Items (i)-(iii) follow immediately from the definitions.  We now prove (iv).  Because of items (ii) and (iii), we need only to check the orthogonality on $\partial_{\Sph^2} U^\Sigma_2$.  Note by \ref{Lconfo} that each level set of $\dbold_0 \circ f$ intersects $\partial_{\Sph^2} \Dp$ orthogonally.  Moreover, by \ref{Ediffsig}, on each such level set $X_\Sigma$ restricts to a rotation of the form $\Rcap^{c \varpi}_\ell$, for some $c \in (0, 1)$.  It follows that for each $q \in \partial_{\Sph^2}U_2$, the vector $X_\Sigma(q)\in \R^3$ lies in the tangent space $T_{X_\Sigma(q)} \Sigma$.  Therefore (iv) holds. 
\end{proof}

\subsection*{Estimates of geometric quantities}
$\phantom{ab}$
\nopagebreak

\begin{definition}
\label{dvhat}
We define $\vhat, \what \in C^\infty(\Dp)$ by requesting that $\vhat = \varpi v$ on $U_1$, $\vhat = 0$ on $U_3$,
\begin{align*}
\vhat  =  \Psibold[\textstyle{\frac{9}{10}}, \frac{8}{10}; \dbold_0 \circ f](\varpi v, 0) 
\quad
\text{on} 
\quad
U_2, 
\quad
\text{and} 
\quad
\what : = \Lcal_{\D} \vhat. 
\end{align*}
\end{definition}

\begin{lemma}
\label{Lvhat}
$\Bcal_{\D} \vhat|_{\partial_{\Sph^2} \Dp}=0$. 
\end{lemma}
\begin{proof}
This follows in the same way as the proof of \ref{Lvmismatch}(ii), so we omit the details. 
\end{proof}

\begin{lemma}[Estimates on the base surfaces]
\label{Lbase}
For $\Sigma$ and $X_\Sigma$ as defined in \ref{dbase}, the following hold. 
\begin{enumerate}[label = \emph{(\roman*)} ]
\item $g - X^*_\Sigma g$ is supported on $U_2$ and satisfies $\| g - X^*_\Sigma g : C^k_{\sym}(U_2, g) \| \leq C(k) \varpi^2$.
\item $\| d\theta^2 - X_\Sigma^{*}g|_{\partial \Sigma} :C^k_{\sym}( \partial_{\Sph^2}U_2, g)\| \leq C(k)\varpi^2$.
\item $\| \partial_{\rr} - X_\Sigma^{*}\eta_\Sigma : C^k_{\sym}( \partial_{\Sph^2}U_2, g) \|
\leq C(k) \varpi^2$.
\item $A_\Sigma$ is supported on $U^\Sigma_2$ and satisfies 
	$\| A_\Sigma : C^k_{\sym}( U^\Sigma_2, g)\| \leq C(k) \varpi$.
\item $H_\Sigma$ is supported on $U^\Sigma_2$ and satisfies 
$ \| H_\Sigma - X^{*-1}_\Sigma \what : C^k_{\sym}(U^\Sigma_2, g)\| \leq C(k)\varpi^2.$
\item For any $u \in C^{k+2}(U_2)$,
	$
	\| \Lcal_\Sigma X^{*-1}_\Sigma u  - X^{*-1}_\Sigma \Lcal_{\D} u : C^k( U^\Sigma_2, g)\| \leq C(k) \varpi^2  \| u : C^{k+2}(U_2, g) \|.
	$
	
\end{enumerate}
\end{lemma}
\begin{proof}
 Lemma \ref{Lbprop}(ii) and (iii) imply that $X_\Sigma$ is an isometry on $\Dp \setminus U_2$, from which the statement on the support of $g - X^*_\Sigma g$ in (i) follows.  The estimate in item (i) follows via a short calculation from using \eqref{Ediffsig} to compute the pullback metric $X^*_\Sigma g$.  Item (ii) follows from (i), and item (iii) follows from (ii). 

The claim about the support of $A_\Sigma$ follows from \ref{Lbprop}(ii)-(iii), using the fact that $P$ and $\Rcap^\varpi_\ell P$ are totally geodesic, and the estimate follows trivially from \eqref{Ediffsig}.
The claim about the support of $H_\Sigma$ also follows from \ref{Lbprop}(ii)-(iii), using the fact that $\Dp$ has vanishing mean curvature.  Using the smallness of $\varpi$ in conjunction with \eqref{Ediffsig}, we find that
\begin{align*}
\| X_\Sigma - X^{\R^3, g}_{U_2, \vhat \nu} : C^k(U_2, g) \| \leq C(k) \varpi^2.
\end{align*}
Item (v) follows from this and the standard expansion for the mean curvature of an immersion into linear and higher-order terms.

Finally, using \cite[Lemma C.10(iv)]{KapMcGLDG} and \cite[Lemma C.11]{KapMcGLDG} to estimate the difference of the linearized operators, we have that
\begin{align*}
\| X_\Sigma^* \Lcal_\Sigma X^{*-1}_\Sigma \utilde - \Lcal_{\D} \utilde : C^k(U_2, g) \| \leq C(k) \| g - X^*_\Sigma g: C^{k+2}(U_2, g)\| \| \utilde : C^{k+2}(U_2, g) \|
\end{align*}
for each $\utilde \in C^{k+2} (U_2) $.  Item (vi) now follows from this using (i) above. 
\end{proof}

\section{Pre-initial Surfaces}
\label{Spreinit}

\subsection*{Catenoidal bridges}
$\phantom{ab}$
\nopagebreak

Consider the function $\phicat:[\tau,\infty)\to\R$  defined by
\begin{equation}
\begin{aligned}
\label{Evarphicat}
\phicat(r):=
\tau\arccosh \frac r \tau &=
\tau\left(\log r-\log \tau+\log\left(1+\sqrt{1-{\tau^2}{r^{-2}}\,}\right)\right)
 \\
&=
\tau\left(\log \frac { 2 r } {\tau} 
+
\log\left(\frac12+\frac12\sqrt{1-\frac{\tau^2}{r^{2}}\,}\right)\right).
\end{aligned}
\end{equation}

By straightforward calculation \cite[Lemma 2.25]{kap} we have the estimate 
\begin{equation} 
\label{Lcatenoid}
\|\, \phicat -
\tau \log ( { 2 r } / {\tau} )
\,
: C^{k}(\,
(9 \tau,\infty)\,,\,
r, d r^2,r^{-2}\,)\,\|
\le 
\, C(k) \, \tau^3. 
\end{equation}

\begin{convention}
\label{conalpha}
We fix now some $\alpha>0$ which we will assume as small in absolute terms as needed. 
We also define $\delta' : = \tau^\alpha$. 
\end{convention}

\begin{definition}[Catenoidal bridges]
\label{dcat}
We define $\cat[p, \tau]\subset \R^3$ to be the catenoid of size $\tau$, centered at $p$, with axis $\R \nu(p)$.   We define the \emph{catenoidal bridge} $\catbr[p, \tau]$ of size $\tau$, centered at $p$, with axis $\R \nu(p)$ by
\[
\catbr[p, \tau] = \cat[p, \tau] \cap D^{\R^3}_p(\delta').
\]
We also define the top half $\catbr^+[p, \tau]$ of $\catbr[p, \tau]$ by $\catbr^+[p, \tau]:= \catbr[p, \tau]\cap \{z \geq 0\}$.
\end{definition}

It will be useful to parametrize $\cat[p, \tau]$ over a portion of a flat cylinder.  In order to do this, we make the following definitions. 

\begin{definition}
\label{dcyl}
Let 
$\cyl : = \R \times \Sph^1 \subset \R\times\R^2$ be the standard cylinder and $(t, \vartheta)$ be the standard coordinates on $\cyl$ 
defined by considering the covering $\Thetacyl:\R^2\to\cyl$ given by 
$
\Thetacyl(t, \vartheta) := (\cos\vartheta,\sin\vartheta, t).
$
We also define $\cyl^+ : = \{  \Thetacyl( \tunder, \vartheta) : |\vartheta| \leq \pi/2\}$
and given $\tunder \in \R$ and $I \subset \R$, we define
\begin{align*}
\begin{gathered}
\cyl_{\tunder} :=  \{ \Thetacyl( \tunder, \vartheta) : \vartheta\in\R \}\subset \cyl,
\quad
\cyl^+_{\underline{t}} : = \cyl_{\tunder} \cap \cyl^+,
\\
\cyl_I := \cup_{\tunder \in I} \cyl_{\tunder},
\quad
\text{and}
\quad
\cyl^+_I :=  \cup_{\tunder \in I} \cyl^+_{\tunder}.
\end{gathered}
\end{align*}
\end{definition}

\begin{definition}
\label{dkhat}
We define $a \in \R_+$ and the embedding $\kappahat : \cyla \rightarrow \cat[\tau, p] \subset \R^3$ by 
\begin{align}
\label{Ea}
a: = \tau^{-1} \phicat(\delta') = \log \frac{2\delta'}{\tau} + O(\tau^{2(1-\alpha)})
\quad 
\text{and} 
\end{align}
 \begin{equation}
\label{Ecatparam}
\begin{aligned}
\kappahat(t, \vartheta) &: = (1, 0 , 0) + \tau(-\cosh t \cos \vartheta, \cosh t \sin \vartheta, t),
\end{aligned}
\end{equation}
where $(t, \vartheta)$ are the standard coordinates for $\cyla$ as in \ref{dcyl}.
\end{definition}

\subsection*{The auxiliary metric}
$\phantom{ab}$
\nopagebreak

\begin{definition}[The auxiliary metric]
\label{dgaux}
Define a metric $g_A$ on $\R^3$ by
\begin{align}
\label{Egaux}
\gaux : = \Omega^2 g, \quad \text{where} \quad
\Omega : = \Psibold \left[ \textstyle{\frac{1}{3}}, \frac{2}{3}; \dbold^{\R^3}_0\right]\left(1, (\dbold^{\R^3}_0)^{-1}\right).
\end{align}
\end{definition}

\begin{lemma}
\label{Lpcatd}
There is a function $\phicataux \in C^\infty ( D_p(4 \delta') \setminus D_p(\delta'/8))$ such that
\begin{align*}
\graph_{D_p(4 \delta') \setminus D_p(\delta'/8)}^{\R^3, \gaux}( \phicataux \nu ) \subset \cat[p, \tau], 
\end{align*}
and moreover the following estimate holds:
\begin{align}
\label{Ephicata}
\| \phicataux - \phicat : C^k ( D_p(4\delta') \setminus D_p(\delta'/8), \dbold_p, g) \| 
\leq 
\tau^3 |\log \tau|^3.
\end{align}
\end{lemma}
\begin{proof}
This follows from combining \ref{Luauxw} with \ref{Lcatenoid}. 
\end{proof}

\subsection*{Pre-initial surfaces}
$\phantom{ab}$
\nopagebreak

\begin{definition}[Pre-initial surfaces]
\label{dpreinit}
We define the pre-initial surface $\Sigmacir : = \Sigmacir[\zeta]$ by 
\begin{align*}
\Sigmacir  : = 
\B^3 \cap \left( \graph^{\R^3, \gaux}_{\Sigma \setminus D_p(\delta')}(\varphigl \nu_\Sigma) \cup \catbr^+[p, \tau]\right)
\end{align*}
where
$\varphigl : \Sigma \setminus D_p(\delta') \rightarrow \R$ is defined by requesting that
\begin{align*}
\begin{gathered}
\varphigl = \Psibold[ 2\delta', 3\delta'; \dbold_p ] ( \phicataux, \phigraph),
\quad 
\text{where}
\quad
\phigraph : = X_\Sigma^{*-1}( \varphi - 
\vhat
 - (\Mcal \varphi) \vunder),
 \end{gathered}
\end{align*}
 $\phicataux$ is as in \ref{Lpcatd}, $\varphi= \varphi\llbracket \zeta; \varpi \rrbracket$ and $\tau = \tau[\zeta, \varpi]$ are as in \ref{dLDf}, $\vhat$ is as in \ref{dvhat}, and $\vunder$ is as in \ref{dobs}.
\end{definition}

\begin{lemma}[Properties of the pre-initial surfaces]
\label{Lpresym}
For $\Mcir$ as defined in \ref{dpreinit}, the following hold.
\begin{enumerate}[label = \emph{(\roman*)} ]
\item  $\Mcir$ is a smooth, $\group$-invariant embedded disk in $\B^3$.
\item $\partial \Mcir \subset \partial \B^3 \cup P \cup \Pprime$.
\item $\Sigmacir \cup \Rcapunder_{P} \Sigmacir$ and $\Sigmacir \cup \Rcapunder_{\Pprime} \Sigmacir$ are smooth disks. 
\end{enumerate}
\end{lemma}
\begin{proof}
Item (i) follows immediately from \ref{dpreinit} and the definitions.  Next, it follows from \ref{dgaux} that $P, \Pprime$, and $\partial \B^3$ are totally geodesic with respect to $\gaux$.  Moreover, it follows from \ref{Lbprop} that $\nu_\Sigma |_{\partial \Sigma \cap \ell}$ is tangent to $\Pprime$ and $\nu_\Sigma |_{\partial_{\Sph^2}\Sigma}$ is tangent to $\partial \B^3$.  From these facts and \ref{dpreinit}, item (ii) follows. 

 Finally, we prove (iii).  That $\Sigmacir \cup \Rcapunder_P \Sigmacir$ is a smooth surface follows immediately from \ref{dpreinit}.  Next, it follows from \ref{Lldexist}(i), \ref{Lbprop}(ii), and \ref{dpreinit} that $\phigraph|_{U^\Sigma_1} : = X^{*-1}_\Sigma (\tau G)$.  From the symmetries of $G$ and the fact that $\Pprime$ is totally geodesic with respect to $\gaux$, the second part of (iii) follows. 
\end{proof}

We write $\Xcir : \Sigmacir \rightarrow \R^3$ for the inclusion map for $\Sigmacir$, $\gcir = \Xcir^* g$ for the induced metric, $\nucir : \Sigmacir \rightarrow \R^3$ for the upward pointing unit normal to $\Sigmacir$, $\Acir$ for the corresponding second fundamental form, $\Hcir$ for the corresponding mean curvature,  and $\Thetacir : \partial_{\Sph^2} \Sigmacir \rightarrow \R$ for the Euclidean inner product
\begin{align}
\label{Etheta}
\Thetacir : = \langle \Xcir , \nucir \rangle|_{\partial_{\Sph^2} \Sigmacir}.
\end{align}

For future applications we introduce coordinates $(\scir, \sigmacir)$ on a neighborhood of $\partial_{\Sph^2} \Sigmacir$ in $\Sigmacir$ small enough so that within it the map of nearest point projection onto $\partial_{\Sph^2} \Sigmacir$ is well-defined and smooth; for any point $p$ in this neighborhood $\sigmacir(p)$ is the distance in $\Sigmacir$ of $p$ from $\partial_{\Sph^2} \Sigmacir$ and $\scir(p)$ is the distance in $\partial_{\Sph^2} \Sigmacir$ of the nearest-point projection of $p$ onto $\partial \Sigmacir$ from an arbitrarily fixed reference point on $\partial_{\Sph^2} \Sigmacir$.  In particular, along $\partial_{\Sph^2} \Sigmacir$, the coordinate vector field $\partial_{\sigmacir}$ is the inward unit conormal for $\Sigmacir$, and $\{ \Xcir_* \partial_{\scir}, \Xcir_* \partial_{\sigmacir}, \nucir \}$ is an orthonormal frame for $\R^3$ along $\partial_{\Sph^2} \Sigmacir$.

It will be useful to define $\rhocir : \Sigmacir \rightarrow \R$ by $\rhocir = \dbold_p \circ \Pi^{\R^3, g}_P$.  Note equivalently that $\rhocir$ is the distance from the axis of $\catbr[p, \tau]$ and in $(t, \vartheta)$ coordinates on $\Sigmacircat$ (recall \ref{dcyl}), $\rhocir = \tau \cosh t$.

\begin{definition}[The catenoidal and graph regions]
\label{dregions}
We define regions $\Sigmacircat$ and $\Sigmacird$ of $\Sigmacir$ by
\begin{align*}
\Sigmacircat : = \Sigmacir \cap D^{\R^3}_{p}(\delta'), 
\quad
\Sigmacird : = \Sigmacir \setminus D^{\R^3}_p(\delta'/8).
\end{align*}
We also define $\picir_\Sigma : \Sigmacird \rightarrow \Sigma$ to be the restriction to $\Sigmacird$ of $\Pi^{g_A}_{\Sigma}$ and $\picir : \Sigmacird \rightarrow \Dp$ by $\picir = X^{-1}_\Sigma \circ \picir_\Sigma$.
\end{definition}

\subsection*{Estimates on $\Sigmacircat$}
$\phantom{ab}$
\nopagebreak

\begin{lemma}
\label{Lcatdiff}
For all small enough $\tau>0$, the map $\kappacir : \cyla^+ \rightarrow \Sigmacircat$ defined by
\begin{align}
\label{Ekappacir}
\kappacir( t, \vartheta) = \kappahat ( t, \lambda(t) \vartheta),
\quad
\text{where}
\quad
\lambda(t): = 1 - \frac{2}{\pi} \arcsin \frac{\tau^2 \cosh^2 t+\tau^2 t^2}{2 \tau \cosh t}
\end{align}
is a diffeomorphism, which moreover restricts to a diffeomorphism of $\partial \cyla^+ \setminus (\cyl^+_0 \cup \cyl^+_a)$ onto $\partial_{\Sph^2} \Sigmacircat$. 
\end{lemma}
\begin{proof}
Straightforward calculation using the definitions. 
\end{proof}

\begin{lemma}[Estimates on $\Sigmacircat$]
\label{Lcatest} 
For $\Mcir$ as in \ref{dpreinit}, the following hold.
\begin{enumerate}[label = \emph{(\roman*)} ]
\item $ \| \kappacir^*\gcir  - \kappahat^*\gcir : C^k( \cyla^+, \rhocir, \kappahat^* \gcir  , \rhocir^4)\| 
\leq C(k)$. 
\item $\| \kappacir^* \Acir - \tau(dt^2 - d\vartheta^2): C^k( \cyla^+,\rhocir,  \kappahat^*\gcir, \rhocir^2)\| 
\leq C(k) \tau$. 
\item $\Hcir|_{\Sigmacircat} = 0$.
\item $\| \Acir : C^k( \Sigmacircat, \rhocir, \gcir ) \| \leq C(k)\tau$.
\item $\| \rhocir : C^k(\Sigmacircat, \rhocir,\gcir, \rhocir)\| + \| \rhocir^{-1} : C^k(\Sigmacircat,\rhocir,  \gcir, \rhocir^{-1})\| \leq C(k)$.
\item $\| \kappacir^* \gcir |_{\partial_{\Sph^2} \Sigmacircat} - \tau^2 \cosh^2 t \,d t^2 : C^k(\{ |\vartheta| = \pi/2\}, \rhocir, \kappahat^* \gcir, \rhocir^4)\|
\leq  C(k)$.
\item $\| \rhocir \partial_{\sigmacir} + ( \sgn \vartheta) \kappacir_* \partial_{\vartheta} : C^k( \partial_{\Sph^2} \Sigmacircat,\rhocir,  \gcir |_{\partial_{\Sph^2} \Sigmacircat}, \rhocir^2)\| \leq C(k).$
\item $\| \Thetacir : C^k(\partial_{\Sph^2} \Sigmacircat,\rhocir,  \gcir) \| \leq C(k) \tau |\log \tau|$. 
\end{enumerate}
\end{lemma}
\begin{proof}
This proof is very similar to parts of the proof of \cite[Prop. 3.29]{KapWiygul}.  It is straightforward to check using \eqref{Ea} and \eqref{Ekappacir} that for each nonnegative integer $k$, that there exists $C(k)> 0$ such that 
\begin{align*}
\| \lambda - 1 : C^k([0,a],  dt^2, \tau \cosh t)\| \leq C(k).
\end{align*}
From this and \eqref{Ekappacir}, it follows that
\begin{align}
\label{Epdiff}
\| \kappacir - \kappahat :C^k( \cyla^+ ,\tau \cosh t, \kappahat^* g,  \tau^2 \cosh^2 t) \| \leq C(k).
\end{align}

Using that
\begin{equation}
\label{Erhocir}
\begin{aligned}
\rhocir ( \kappahat ( t, \vartheta)) &= \tau \cosh t, 
\quad \text{for} \quad (t, \vartheta) \in \kappahat^{-1} (\Sigmacircat) \subset \cyla^+,\\
\rhocir (\kappacir ( t, \vartheta)) &= \tau \cosh t, 
\quad \text{for} \quad (t, \vartheta) \in \cyla^+,\\
\kappahat |^*_{\kappahat^{-1}(\Sigmacircat)} \gcir &= \tau^2 \cosh^2 t (dt^2 + d\vartheta^2), 
\quad \text{and}\\
\kappahat |^*_{\kappahat^{-1}(\Sigmacircat)}\Acir &= \tau( d t^2 - d \vartheta^2),
\end{aligned}
\end{equation}
in combination with \eqref{Epdiff}, we conclude items (i) and (ii).  Item (iii) is clear, since $\Sigmacircat$ is a subset of a Euclidean catenoid. 
  Item (iv) follows from (ii) and (i).  Item (v) follows from \eqref{Erhocir}, item (i), and the definitions. 

We now turn to the boundary geometry.  Item (vi) follows directly from (i), since $\kappacir ( \{ |\vartheta| = \pi/2\}) = \partial_{\Sph^2} \Sigmacircat$.  Then, item (vii) follows from (vi). 

  Finally, 
we estimate $\Thetacir$ on $\partial_{\Sph^2} \Sigmacircat$.  It follows from \eqref{Ecatparam} that
\begin{align*}
\Xcir( \kappahat(t, \vartheta)) &= (1 - \tau \cosh t \cos \vartheta, \tau \cosh t \sin \vartheta, \tau t), \\
\nucir (\kappahat(t, \vartheta)) &= (\sech t \cos \vartheta , - \sech t \sin \vartheta, \tanh t ),
\end{align*}
from which it follows that
\begin{align*}
\langle \Xcir( \kappahat(t, \vartheta)) , \nucir ( \kappahat (t,\vartheta))\rangle 
= \sech t \cos \vartheta - \tau + \tau t \tanh t.
\end{align*}
Evaluating along $\partial_{\Sph^2} \Sigmacircat$ and using \ref{Lcatdiff}, we conclude that
\begin{align}
\kappacir^* \Thetacir (t) = \frac{1}{2} \tau t^2 \sech^2 t - \frac{1}{2} \tau + \tau t \tanh t.
\end{align}
In combination with item (vi), this concludes the proof of the estimate of $\Thetacir$. 
\end{proof}

\subsection*{Estimates on $\Sigmacird$}
$\phantom{ab}$
\nopagebreak

\begin{lemma}[The gluing region]
\label{Lglureg}
For $\Mcir$ as defined in \ref{dpreinit}, the following hold. 
\begin{enumerate}[label = \emph{(\roman*)} ]
\item $\| \varphigl  : C^{k}_{\sym}(  D_p(4\delta')\setminus D_p(\delta') ,  (\delta')^{-2} g) \| 
\leq C(k) \tau |\log \tau|$.
\item $\| \gcir - \picir^* g : C^k_{\sym}(D_p(4\delta')\setminus D_p(\delta') ,  (\delta')^{-2} g)\| \leq C(k) \tau^2 |\log \tau|^2$. 
\item $\| \picir^{*-1}\Hcir : C^{k}_{\sym}( D_p(3\delta')\setminus D_p(2\delta'), (\delta')^{-2} g, (\delta')^{-2}) \|
 \leq C(k) \tau^{1+\alpha}|\log \tau|$.
 \item $\| \picir^{*-1}\Thetacir: C^{k}_{\sym}( \partial_{\Sph^2}(D_p(3 \delta') \setminus D_p(2 \delta')), (\delta')^{-2} g) \|
 \leq C(k)\tau |\log \tau|$.
\end{enumerate}
\end{lemma}
\begin{proof}
On $U : = D_p(4\delta')\setminus D_p(\delta')$, we have using \ref{dpreinit}, the smallness of $\delta'$, \ref{dobs}, and \ref{Lbase} that
\begin{align}
\label{Ephigl}
\varphigl = 
\varphi - (\Mcal \varphi) \vunder + \Psibold[2\delta', 3\delta'; \dbold_p](\phicataux - \varphi+ (\Mcal \varphi) \vunder, 0).
\end{align}
 By scaling the ambient metric to $\gtilde : = (\delta')^{-2} g$ and expanding in linear and higher order terms we have
\begin{align*}
(\delta')^2 \picir^{*-1}\Hcir  = \Delta_{\gtilde} \varphigl + \delta' \Qtilde_{(\delta')^{-1} \varphigl}.
\end{align*}
Using \eqref{Ephigl}, we clearly have
\begin{align*}
\|\varphigl \| &\leq C(k) \left( \tau |\log \tau| + \|\phicataux - \varphi+ (\Mcal \varphi) \vunder\| \right),
\end{align*}
where in this proof when we do not specify the norm we mean the $C^{k}_{\sym}(U, (\delta')^{-2} g)$ norm.  For any $k \geq 2$, we have also
\begin{equation*}
\begin{aligned}
\| \Delta_{\gtilde}  \varphigl : C^{k-2}_{\sym}(U, (\delta')^{-2} g) \| 
&\leq C(k) \| \phicataux - \varphi+ (\Mcal \varphi) \vunder \|, \\
\| \delta' \Qtilde_{(\delta')^{-1} \varphigl} C^{k-2}_{\sym}(U, (\delta')^{-2} g)  \| 
&\leq C(k) (\delta')^{-1} \| \varphigl \|^2.
\end{aligned}
\end{equation*}
 We conclude that if $\| \varphigl : C^{k+2}(U, (\delta')^{-2} g) \| \leq \delta'$ (to control the quadratic terms), then we have
\begin{align*}
\| (\delta')^2\picir^{*-1}\Hcir  : C^{k-2}_{\sym}(U, (\delta')^{-2} g)\| 
\leq 
\, C(k)\, 
\left(\, 
 (\delta' )^{-1} \tau^2 |\log \tau|^2
+
 \| \phicataux - \varphi+ (\Mcal \varphi) \vunder\|\right).
\end{align*}
On $U$ we have $\phicataux - \varphi+ (\Mcal \varphi) \vunder = (I) +(II) + (III) + (IV)$, where
\begin{equation}
\begin{gathered}
(I): = \phicataux - \phicat\circ \dbold_p, \quad
(II) : =  \phicat \circ \dbold_p - \tau \log \frac{2\dbold_p}{\tau} , \\
(III):=  \tau \log \frac{2\dbold_p}{\tau} - \varphi + \Mcal \varphi, \quad 
(IV)  :=  (\Mcal \varphi) (\vunder -1).
\end{gathered}
\end{equation}
Using \eqref{Ephicata} to estimate $(I)$, \eqref{Lcatenoid} to estimate $(II)$, \ref{Lmismatch}(iii) to estimate $(III)$, and \eqref{Etau}, \ref{Lpmis}, and \ref{Lvmismatch} to estimate $(IV)$, we have
\begin{equation}
\label{Eii}
\begin{gathered}
 \| (I) \| \leq C(k) \tau^3 |\log \tau|^3, \quad
 \| (II) \| \leq C(k) \tau^{3-2\alpha},\\
 \| (III) \| \leq C(k) \tau^{1+\alpha}|\log \tau|, 
 \quad
 \| (IV) \| \leq C(k) |\Mcal \varphi| \tau^\alpha \leq C(k)\tau^{1+ \alpha}.\\
 \end{gathered}
 \end{equation}
 The proof of item (i) is completed by combining the estimates above.  Next,  Computing the pullback metric $\picir^{*-1} \gcir$ via \ref{Egauxexp} and estimating we find that
 \begin{align*}
 \| \gcir - \picir^* g : C^k_{\sym}(U, (\delta')^{-2}g)\| 
 &\leq C(k) \| \varphigl : C^{k+1}_{\sym}(U, (\delta')^{-2} g)\|^2.
 \end{align*}
Item (ii) then follows from this using (i).
 
 Next, the estimate (iii) follows from the estimates above.  Finally, by using the identity for the angle function from \cite[Lemma 5.6]{KapWiygul}, we have that
 \begin{align*}
\|  \picir^{*-1}\Thetacir : C^k_{\sym}( \partial_{\Sph^2} U, (\delta')^{-2} g)\| 
\leq 
C(k)
\| \varphigl : C^{k+1}_{\sym}(U, (\delta')^{-2} g)\|.
 \end{align*}
Item (iv) follows from this using (i). 
\end{proof}

\begin{lemma}[Estimates on $\Sigmacird$]
\label{Lscird}
For $\Sigmacir$ as defined in \ref{dpreinit}, the following hold.
\begin{enumerate}[label = \emph{(\roman*)} ]
\item $\| \gcir - \picir^* g : C^k_{\sym}(\Sigmacird, \rhocir, \gcir)\| \leq C(k)\tau^2 | \log \tau|^2$.
\item $\| \gcir|_{\partial \Sigmacir} - \picir^* d\theta^2 : C^k_{\sym}( \partial_{\Sph^2} \Sigmacird, \rhocir, \gcir)\|
\leq C(k)\tau^2 | \log \tau|^2$. 
\item $\| \partial_{\sigmacir} + \picir^{-1}_*\partial_{\rr} : C^k_{\sym}(\partial_{\Sph^2} \Sigmacird, \rhocir, \gcir)\| 
\leq C(k) \tau^2 | \log \tau|^2$. 
\item $\| \Acir - \picir^* A_\Sigma : C^k( \Sigmacird, \rhocir, \gcir ) \| \leq C(k)\tau|\log \tau|$.
\item $\| \Hcir - (\Mcal \varphi) w \circ \picir : C^k_{\sym}( \Sigmacird, \rhocir, \gcir, \rhocir^{-2}) \| \leq C(k)\tau^{1+\alpha} | \log \tau|$.
\item $\|\Thetacir : C^k_{\sym}(\partial_{\Sph^2} \Sigmacird, \rhocir, \gcir)\|\leq C(k) \tau | \log \tau |$ and $\|\Thetacir : C^k_{\sym}(\partial_{\Sph^2} \Sigmacird \cap \{\rhocir > 3\delta'\}, \rhocir, \gcir)\|\leq C(k) \tau^3 | \log \tau |^3$.
\end{enumerate}
\end{lemma}
\begin{proof}

On the gluing region, the estimate in (i) follows from \ref{Lglureg}(ii).  On the other hand, by an  argument analogous to the one in the proof of \ref{Lglureg}(ii), it follows that
\begin{align*}
\| \gcir - \picir^*_\Sigma g : C^k(\Sigmacird, \rhocir, \gcir) \|
\leq C(k) \tau^2 |\log \tau|^2,
\end{align*}
where the estimate follows by estimating $\phigraph$ as defined in \ref{dpreinit} using \ref{Lmismatch} and \eqref{Etau}.  The estimate in (i) now follows by combining the preceding with \ref{Lbase}(i).  Item (ii) follows immediately from item (i), and item (iii) follows from item (ii). 

We now estimate $\Acir$.  By straightforward estimates using the definitions, we have that
\begin{align*}
\| \Acir - \picir^* A_\Sigma : C^k(\Mcird, \rhocir, \gcir)\| &\leq C(k) \| \varphigl : C^{k+2}(\picir(\Mcird), \rhocir, \gcir)\|\\
&\leq C(k) \tau |\log \tau |.
\end{align*}
This completes the proof of (iv).

We now estimate the mean curvature.  First note that on $\picir^{-1}(D_p(3\delta') \setminus D_p(2\delta'))$, the estimate in (v) follows from \ref{Lglureg}(iii).  We next estimate $\Hcir$ on $\picir^{-1}_{\Sigma}(U^\Sigma_2)$. 
Expanding in linear and higher order terms, we have
\begin{align*}
\picir_\Sigma^{*-1} \Hcir = H_\Sigma + \Lcal_\Sigma \phigraph + Q_{\phigraph}, 
\quad
\text{where}
\quad
\| Q_{\phigraph} : C^k_{\sym}(U^\Sigma_2, g) \|
 \leq C(k) \| \phigraph : C^{k+2}_{\sym}(U^\Sigma_2, g)\|^2. 
\end{align*}

By combining the definition of $\phigraph$ in \ref{dpreinit} with \ref{dLD}, \ref{dobs}, and \ref{dvhat}, it follows that $\Lcal_{\D} X^*_\Sigma = - \what$ on $U_2$.  By adding and subtracting $X^{*-1}_{\Sigma} \what$ to $H_\Sigma + \Lcal_\Sigma \phigraph$ and 
using \ref{Lbase}(v)-(vi) to estimate, we conclude that
\begin{align*}
\| H_\Sigma + \Lcal_\Sigma \phigraph : C^k_{\sym}(U^\Sigma_2, g)\|
\leq 
C(k)
\varpi^2.
\end{align*}
By combining with the estimate on the quadratic terms above, we have that
\begin{align*}
\| \picir^{*-1}_\Sigma \Hcir : C^k_{\sym}(U^\Sigma_2, g)\| \leq C(k) \varpi^2 \leq C(k) \tau^2 | \log \tau|^2.
\end{align*}
This proves (v) on $\picir^{-1}(U^\Sigma_2)$.  The proof of (v) on the rest of $\Sigmacird$ is very similar to, but easier than the proof of the estimate on $\picir^{-1}(U_2)$ just completed, so we omit the details.

Finally, we estimate the boundary angle.  Because of \ref{Lcatest}(viii) and \ref{Lglureg}(iv), it  suffices to estimate $\Thetacir$ on $\partial_{\Sph^2} \Ucir$, where $\Ucir := \picir^{-1}(\Dp \setminus D_0(3\delta'))$.  For convenience, we also denote $U_\Sigma : = \pi_\Sigma(\Ucir)$.  By combining the formula for the boundary angle in \cite[Lemma 5.6]{KapWiygul} with \cite[Lemma 5.19]{KapWiygul}, we find that
\begin{align*}
\| \Thetacir : C^k_{\sym}( \partial_{\Sph^2} \Ucir, \rhocir, \gcir)\|
\leq
C(k)
\|\Bcal_{\Sigma}\phigraph|_{\partial_{\Sph^2} \Sigma}  : C^k_{\sym}( \partial_{\Sph^2} U_\Sigma, \rhocir, \gcir) \|.
\end{align*}
On the other hand, combining the definition of $\phigraph$ in \ref{dpreinit} with \ref{dLD}(iv), \ref{Lvmismatch}(ii), and \ref{Lvhat} establishes that
\begin{align*}
\Bcal_{\D} X^*_\Sigma \phigraph|_{\partial_{\Sph^2} \picir(\Ucir)} = 0.
\end{align*}
From this and the definition of $\phigraph$ in \ref{dpreinit}, it follows that
\begin{align*}
\| \Bcal_{\Sigma}\phigraph|_{\partial_{\Sph^2} \Sigma}  : C^k_{\sym}(\partial_{\Sph^2} U_\Sigma, \rhocir, \gcir) \|
&\leq 
C(k) 
\| \partial_{\rr} - X^*_{\Sigma} \eta_\Sigma : C^k_{\sym}( \partial_{\Sph^2} U_\Sigma, \rhocir, \gcir)\| 
\| X^*_\Sigma \phigraph : C^{k+1}_{\sym}(\picir(\Ucir),   \rhocir, \gcir )\|
\\
&\leq 
C(k) \tau^3|\log \tau|^3,
\end{align*}
where the last estimate uses \ref{Lbase}(iii).  The proof of (vi) is now completed by combining the preceding estimates. 
\end{proof}

\subsection*{Estimates on $\Sigmacir$}
$\phantom{ab}$
\nopagebreak

\begin{lemma}[Estimates on $\Sigmacir$]
\label{Lmcirest} 
Let $\Sigmacir$ be as in \ref{dpreinit}.  There exists $\epsilon>0$ such that the following hold. 
\begin{enumerate}[label = \emph{(\roman*)} ]
\item $(\scir, \sigmacir)$ is a $C^k$ coordinate system on $\{ \sigmacir < \epsilon \rhocir\}$. 
\item $\| d\scir : C^k( \{ \sigmacir < \epsilon \rhocir \}, \rhocir, \gcir, \rhocir)\| 
+
\| \sigmacir : C^k(  \{ \sigmacir < \epsilon \rhocir \}, \rhocir, \gcir, \rhocir)\| \leq C(k).$
\item $\| \Acir : C^k_{\sym}( \Sigmacir, \rhocir, \gcir ) \| \leq C(k)\tau | \log \tau|$.
\item  $\| \Hcir - (\Mcal \varphi) w \circ \picir : C^k_{\sym}( \Sigmacir, \rhocir, \gcir, \rhocir^{-2}) \| \leq C(k)\tau^{1+\alpha} | \log \tau|$.
\item $\|\Thetacir : C^k_{\sym}(\partial_{\Sph^2} \Sigmacir, \rhocir, \gcir)\|\leq C(k) \tau | \log \tau |$ and $\|\Thetacir : C^k_{\sym}(\partial_{\Sph^2} (\Sigmacir \cap \{\rhocir > 3 \delta'\}), \rhocir, \gcir)\|\leq C(k) \tau | \log \tau |$.
\end{enumerate}
\end{lemma}
\begin{proof}
 Items (i) and (ii) follow as in the proof of \cite[Prop. 3.29(xv)-(xvi)]{KapWiygul}, so we omit the details.  Item (iii) follows from combining \ref{Lbase}(iv), \ref{Lcatest}(iv) and \ref{Lscird}(iv).  Item (iv) follows by combining \ref{Lcatest}(iii) and \ref{Lscird}(v).  Finally, item (v) follows from \ref{Lcatest}(viii) and \ref{Lscird}(vi).   
\end{proof}

\section{The initial surfaces}
\label{Sinit}
\nopagebreak

We now bend $\Mcir$ near the spherical part $\partial_{\Sph^2} \Mcir$ of its boundary to make it intersect the sphere $\Sph^2$ orthogonally.  We do this by picking a small function on $\Mcir$ whose graph has the desired property. 

\begin{definition}[Initial surfaces]
\label{dinit}
We define the initial surface $M = M[\zeta]$ to be the image of the deformation  $\Xcir_{\ucir}[\zeta] : \Mcir[\zeta] \rightarrow \R^3$ defined by
\begin{align}
\label{EMinit}
\Xcir_{\ucir} = \Xcir(p) + \ucir (p) \nucir (p),
\end{align}
where $\ucir \in C^\infty_{\sym}(\Mcir)$ is defined by
\begin{align}
\label{Eucir}
\ucir ( \scir, \sigmacir) = - \sigmacir  \frac{\Thetacir(\scir)}{\sqrt{1- \Thetacir^2(\scir)}} \Psibold\left[\frac{\epsilon}{4}, \frac{\epsilon}{2}; \rhocir^{-1} \sigmacir\right](1, 0),
\end{align}
and $\epsilon$ is as in \ref{Lmcirest}, so that the coordinates $(\scir, \sigmacir)$ are well-defined and smooth on the support of the cutoff function above and we understand that $\ucir$ identically vanishes elsewhere. 
\end{definition}

We write $X : M \rightarrow \R^3$ for the inclusion map of $M$ in $\R^3$, $g$ for the induced Euclidean metric, $\nu : M\rightarrow \R^3$ for the upward pointing unit normal on $M$, $A$ for the second fundamental form, and $H$ for the mean curvature.  We also define 
\begin{equation}
\begin{gathered}
\rho : = \Xcir_{\ucir}^{-1*} \rhocir, 
\quad
\Sigmacat = \Xcir_{\ucir} (\Sigmacircat), 
\quad
\Sigmad = \Xcir_{\ucir} (\Mcird), 
\quad
\pi : = \picir \circ \Xcir_{\ucir}^{-1}.
\end{gathered}
\end{equation}

\begin{lemma}[Properties of the initial surfaces]
\label{Lpm}
For $M$ as defined in \ref{dinit}, the following properties hold.
\begin{enumerate}[label = \emph{(\roman*)} ]
\item $M$ is a smooth, $\group$-invariant embedded disk in $\B^3$. 
\item $\partial M \subset \partial \B^3 \cup P \cup \Pprime$. 
\item $M\cup \Rcapunder_P M$ and $M \cup \Rcapunder_{\Pprime} M$ are smooth disks. 
\item $M$ meets $\partial \B^3$ orthogonally along $\partial_{\Sph^2} M$.
\end{enumerate}
\end{lemma}
\begin{proof}
Items (i)-(ii) follow immediately from \ref{dinit} and corresponding properties for $\Sigmacir$ from \ref{Lpresym}.  Finally, the proof that the initial surface $M$ meets $\partial \B^3$ orthogonally is identical to the proof of the analogous fact in \cite{KapWiygul}, so we omit the details. 
\end{proof}

\begin{lemma}[Estimates on $M$]
\label{Pinit}
The following estimates hold. 
\begin{enumerate}[label = \emph{(\roman*)} ]
\item $\| \ucir : C^k(\Sigmacir, \rhocir, \gcir, \rhocir)\| \leq C(k) \tau |\log \tau|$ and $\| \ucir : C^k(\Sigmacir\cap \{\rhocir > 3\delta'\}, \rhocir, \gcir, \rhocir)\| \leq C(k) \tau^3 |\log \tau|^3$.
\item $\| \kappa^* g - \kappahat^* \gcir : C^k(\cyla^+, \rho, \kappacir^*\gcir, \rhocir \tau^2 | \log \tau |^2 + \rhocir^4)\| \leq C(k)$. 
\item $\| g - \pi^* g : C^k(\Sigmad, \rho, g)\| \leq C(k) \tau^2 | \log \tau|^2.$
\item $\|A : C^k(M, \rho, g)\| \leq C(k)\tau |\log \tau|$ and $\| \rho^2 |A|^2|_{\Sigmacat} - 2 \tau^2\rho^{-2} : C^k(\Sigmacat, \rho, g) \| \leq C(k)\tau | \log \tau|$.
\item $\| H - (\Mcal \varphi) w \circ \pi : C^k(M, \rho, g, \rho^{-2} )\| \leq C(k)\tau^{1+\alpha} |\log \tau|$. 
\item $\| \rho \partial_\sigma + (\sgn \vartheta) \kappa_* \partial_{\vartheta} : C^k(\partial_{\Sph^2} \Sigmacat, \rho, g, \rho^2 + \tau| \log \tau|^2)\| \leq C(k)$.
\item $\| \partial_{\sigma} + \pi^{-1}_* \partial_{\rr} :C^k(\Sigmad, \rho, g)\| \leq C(k) \tau^2 |\log \tau|^2$.
\end{enumerate}
\end{lemma}
\begin{proof}

This proof is very similar to the proof of \cite[Prop. 4.7]{KapWiygul}.  Item (i) follows from estimating \eqref{Eucir} using \ref{Lmcirest}.

It follows by arguing as in the proof of \cite[Prop. 4.7]{KapWiygul} that 
\begin{align*}
\| \Xcir^*_{\ucir}g - \gcir : C^k(\Mcir, \rhocir, \gcir, \rhocir)\| 
&\leq
C(k) \| \ucir \Acir : C^k(\Mcir, \rhocir, \gcir, \rhocir)\|\\
&\leq C(k) \tau^2 |\log \tau|^2,
\end{align*}
where the second inequality uses the estimate on $\ucir$ in (i) and the bound on $\Acir$ in \ref{Lmcirest}(iii).  
In combination with \ref{Lcatest}(i), this proves item (ii).  In combination with \eqref{Enormmult} and  \ref{Lscird}(ii), the previous estimate implies (iii).  Item (vi) then follows immediately from (ii), and (vii) follows from (iii).

It also follows by arguing as in the proof of  \cite[Prop. 4.7]{KapWiygul} that
\begin{equation}
\label{EAdiff}
\begin{aligned}
\| \Xcir^*_{\ucir} A - \Acir : C^k(U, \rhocir, \gcir, \rhocir)\| 
&\leq 
C(k) \| \ucir : C^{k+2}(U, \rhocir, \gcir, \rhocir) \|
\end{aligned}
\end{equation}
where $U$ is either $\Sigmacir$ or $\Sigmacir \cap \{\rhocir >3 \delta'\}$.   By using this in conjunction with \ref{Lmcirest}(iii), (i)-(ii) above, the estimates in item (iv) follow. 

It follows from \eqref{EAdiff} that
\begin{align}
\| \Xcir^*_{\ucir} H - \Hcir : C^k(\Sigmacir, \rhocir, \gcir, \rhocir^{-1})\| 
\leq
C(k) \| \ucir : C^{k+2}(U, \rhocir, \gcir, \rhocir) \|, 
\end{align}
where as before, $U$ is either $\Sigmacir$ or $\Sigmacir \cap \{ \rhocir > 3\delta'\}$. 
Item (v) follows from this, item (i) above, and \ref{Lmcirest}(iv).
\end{proof}

\subsection*{Graphs over the initial surfaces}
$\phantom{ab}$
\nopagebreak

If $\phi \in C^1_{\sym}(M)$ is appropriately small, we define the perturbation $M_\phi$ of $M$ by $\phi$ by
\begin{align}
\label{EMpert}
M_\phi : = \graph^{\R^3, \gaux}_M(\phi \nu ) = X^{\R^3, \gaux}_{M, \phi \nu}(M),
\end{align}
where $X^{\R^3, \gaux}_{M, \phi \nu}$ is as in \ref{NT} and $\nu$ is the Euclidean unit normal to $M$.

\begin{lemma}[Properties of $M_\phi$]
\label{LMpertang}
Suppose $M_\phi$ is as in \ref{EMpert}.  Then the following hold. 
\begin{enumerate}[label = \emph{(\roman*)} ]
\item $\partial M_\phi \subset \partial \B^3 \cup P \cup \Pprime$. 
\item $M_\phi$ meets $\partial \B^3$ orthogonally iff $\Bcal_M\phi|_{\partial_{\Sph^2} M_\phi} = 0$.
\item $M_\phi$ meets $P \cup \Pprime$ orthogonally iff $\partial_{\eta} \phi|_{\partial M \setminus \partial_{\Sph^2} M} = 0$.
\end{enumerate}
\end{lemma}
\begin{proof}
Since $M$ meets $\partial \B^3 \cup P \cup \Pprime$ orthogonally along $\partial M$ by \ref{Lpm}, and $\partial \B^3$, $P$, and $\Pprime$ are totally geodesic with respect to $\gaux$ (recall \ref{dgaux}), it follows from \eqref{EMpert} that (i) holds.  Items (ii) and (iii) follow by combining \cite[Lemma 5.6]{KapWiygul} and \cite[Lemma 5.19]{KapWiygul}.
\end{proof}

\section[The linearized equation on the initial surfaces]{The linearized equation on the initial surfaces}
\label{S:linearized}

\subsection*{Global norms and the mean curvature on the initial surfaces}
$\phantom{ab}$
\nopagebreak

\begin{definition}
\label{dnorm}
For $k\in \N$, $\betahat \in (0, 1)$, $\gammahat \in \R$, and $U$ a domain in $\Dp$, $M$, or $\cyl^+$, we define
\begin{align*}
\| u \|_{k, \betahat, \gammahat; U} : = 
\| u : C^{k, \betahat}(U, r, g, r^{\gammahat})\|,
\end{align*}
where $r : = \dbold_p$ and $g$ is the standard metric on $\D$ when $U \subset \D$, $\rr: = \rho$ and $g$ is the induced Euclidean metric on $M$ when $U \subset M$, and $r: =  \cosh t$ and $g$ is the metric $\cosh^2 t (dt^2 + d\vartheta^2)$ when $U \subset \cyl^+$.
\end{definition}

\begin{convention}
\label{cholder}
From now on, we fix some $\beta \in (0, 1)$ and fix $\gamma = \frac{1}{2}$.  We will suppress the dependence of various constants on $\beta$ and $\gamma$. 
\end{convention}

We estimate now the mean curvature on an initial surface $M$ in terms of the global norm defined in \ref{dnorm}. 

\begin{lemma}
\label{LglobalH}
$\| H - (\Mcal \varphi) w \circ \pi \|_{0, \beta, \gamma-2; M} \leq \tauunder^{1+\alpha/3}$.
\end{lemma}
\begin{proof}
The estimate follows from \ref{dnorm} and \ref{Pinit}(v), using from \ref{cholder} that $\gamma = 1/2$.
\end{proof}

In the next lemma, we compare the linearized operator $\Lcal_M$ to corresponding model operators on $\cyl$ and on $\D$.  In order to do this, we need the following. 

\begin{definition}
We define the operator $\Lcalhatcat$ on $\cyl$ by $\Lcalhatcat : = \partial^2_t + \partial^2_{\vartheta} + 2 \sech^2 t$.
\end{definition}

\begin{lemma}
\label{Lopdiff}
The following hold.
\begin{enumerate}[label = \emph{(\roman*)} ]
\item If $f \in C^{2, \beta}(\kappa^{-1}(\Sigmacat))$, then
	 \begin{enumerate}[label=\emph{(\alph*)}]
	\item $\| \Lcal_M \kappa^{*-1} f - \kappa^{*-1}\rho^{-2} \Lcalhatcat f \|_{0, \beta, \gamma-2; \Sigmacat}
	\leq
	 C 
	\tau^{2\alpha} \| f\|_{2, \beta, \gamma; \kappa^{-1}(\Sigmacat)}$.
	\item 
	$
	 \| \partial_\sigma \kappa^{*-1} f + \kappa^{*-1}\rho^{-1}(\sgn \vartheta) \partial_{\vartheta} f\|_{0, \beta, \gamma-1; \Sigmacat}
	 \leq 
	 C
	 \tau^{2\alpha}
	 \|f \|_{1, \beta, \gamma; \kappa^{-1}(\Sigmacat)}.$
	 \end{enumerate}
\item If $f \in C^{2,\beta}(\pi(\Sigmad))$, then 
	 \begin{enumerate}[label=\emph{(\alph*)}]
	\item 
	$
	\|\Lcal_M \pi^* f - \pi^* \Lcal_{\D} f \|_{0, \beta, \gamma-2; \Sigmad}
	\leq
	C
	\tau^{2(1-\alpha)} |\log \tau|^2
	\| f\|_{2, \beta, \gamma; \pi( \Sigmad)}.
	$
	\item 
	$
	\| \partial_\sigma \pi^* f + \pi^* \partial_{\rr} f \|_{0, \beta, \gamma-1; \partial_{\Sph^2} \Sigmad}
	\leq
	C
	\tau^{2(1-\alpha)}|\log \tau|^2
	\| f\|_{1, \beta, \gamma; \pi(\partial_{\Sph^2} \Sigmad)}.
	$
	\end{enumerate}
\end{enumerate}
\end{lemma}
\begin{proof}
We first prove (i).  For the estimate (i)(a), by the definitions it suffices to prove that
\begin{align*}
\| \rho^2 \Lcal_M \kappa^{*-1} f - \kappa^{*-1} \Lcalhatcat f \|_{0, \beta, \gamma; \Sigmacat}
\leq 
C
\tau^{2\alpha}  \| f\|_{2, \beta, \gamma; \kappa^{-1}(\Sigmacat)}. 
\end{align*}
The estimate now follows by using \cite[Lemma C.10(iv)]{KapMcGLDG} and \ref{Pinit}(ii) to estimate the difference of the corresponding Laplacians, and using \ref{Pinit}(iv).  The proof of (i)(b) is similar and uses \ref{Pinit}(vii), so we omit the details.

We now prove (ii).  Using \cite[Lemma C.10(iv)]{KapMcGLDG} to estimate the difference of the Laplacians, \cite[Lemma C.11]{KapMcGLDG} to estimate norm squared of the second fundamental form of $M$, we find that
\begin{align*}
\|\Lcal_M \pi^* f - \pi^* \Lcal_{\D} f \|_{0, \beta, \gamma-2; \Sigmad}
&\leq
C
\|\rho^2 \Lcal_M \pi^* f - \rho^2\pi^* \Lcal_{\D} f \|_{0, \beta, \gamma; \Sigmad}\\
&\leq C
\| \rho^{-2}( g - \pi^*g)\|_{2, \beta, \gamma; \Sigmad} \| f \|_{2, \beta, \gamma; \pi(\Sigmad)}.
\end{align*}
The desired estimate now follows from using \ref{Pinit}(iii) to estimate the difference of the metrics.  The proof of (ii)(b) is similar and uses \ref{Pinit}(vi), so we omit the details. 
\end{proof}

\subsection*{The definition of $\Rcal^{appr}_M$}
$\phantom{ab}$
\nopagebreak

\begin{definition}[Localization operators]
\label{dloc}
We define operators
\begin{equation*}
\begin{aligned}
&\Psicat : C(\Sigmacat) \rightarrow C(M) \quad \text{and} \quad
\Psicat : C(\partial_{\Sph^2}\Sigmacat)\rightarrow C(\partial M), \\
&\Psidisk : C(\Sigmadisk) \rightarrow C(M) \quad \text{and} \quad 
\Psidisk: C(\partial_{\Sph^2} \Sigmadisk) \rightarrow C(\partial M)
\end{aligned}
\end{equation*}
by requesting that $\Psicat E$ and $\Psidisk E$ are respectively supported on $\Sigmacat$ and $\Sigmad$ and satisfy
\begin{equation}
\begin{aligned}
&\Psicat E : = \kappa^{*-1} ( \Psibold[ a-2 , a-1 ; t]( \kappa^* E, 0) ), \\
& \Psidisk E : = \pi^* ( \Psibold[\textstyle{\frac{1}{8}}\delta', \frac{1}{7} \delta'; \dbold_p]( 0 , \pi^{*-1} E)).
 \end{aligned}
 \end{equation}
\end{definition}

\begin{definition}
For $S = \cyl^+$ or $S = \partial \cyl^+$, we define the Banach spaces
\begin{align*}
C^{k, \beta, \gamma}_{\sym}(S) = \{ u \in C^{k, \beta}(S, g): \| u \|_{k, \beta, \gamma; S} < \infty,  u(t, \vartheta) = u(t, - \vartheta) \}.
\end{align*}
\end{definition}

We recall the following result from \cite[Prop. 6.9]{KapWiygul}: 
\begin{prop}[Solvability of the model problem on the half-catenoid]
\label{pclin}
There is a linear map
\begin{align}
\Rcalhat_{\cat} : C^{0, \beta, \gamma}_{\sym}(\cyl^+) \times  C^{1, \beta, \gamma}_{\sym}(\partial \cyl^+) 
\rightarrow  C^{2, \beta, \gamma}_{\sym}(\cyl^+)
\end{align}
such that if $(E, E^\partial)$ belongs to the domain of $\Rcalhat_{\cat}$ and $u = \Rcalhat_{\cat} (E, E^\partial)$, then the following hold.
\begin{enumerate}[label = \emph{(\roman*)} ]
\item $\Lcalhatcat u = E$. 
\item $-(\sgn \vartheta) \partial_\vartheta u|_{\partial \cyl^+} = E^{\partial}$.
\item $\| u \|_{2, \beta, \gamma; \cyl^+} \leq
C( \| E\|_{0, \beta, \gamma; \cyl^+} + \| E^{\partial}\|_{1, \beta, \gamma; \partial \cyl^+})$. 
\end{enumerate}
\end{prop}

We next consider a corresponding model problem on the half-disk.  Because of \ref{pclin}, it will suffice to assume that the inhomogeneous data on the half-disk is supported outside a neighborhood of $p$ whose size is on the order of $\delta'$. 

\begin{prop}[Solvability of the model problem on the half-disk]
\label{pdlin}
There is a linear map
\begin{multline*}
\Rcalhat_{\D} : \{ E \in C^{0, \beta}_{\sym}(\Dp) : \supp(E) \subset \Dp \setminus D_p(\delta'/8)\}
\\
\times \{ E^\partial \in C^{1, \beta}_{\sym} (\partial_{\Sph^2} \Dp) : \supp(E^\partial) 
\subset \partial_{\Sph^2} \Dp \setminus D_p(\delta'/8)\} 
\rightarrow
C^{2, \beta}_{\sym}(\Dp) \times \R
\end{multline*}
such that if $(E, E^\partial)$ belongs to the domain of $\Rcalhat_{\D}$ and $(u, \mu) = \Rcalhat_{\D}(E, E^\partial)$, then the following hold.
 \begin{enumerate}[label = \emph{(\roman*)} ]
\item $\Lcal_{\D} u = E+ \mu w$.
\item $\Bcal_{\D} u|_{\partial_{\Sph^2} \Dp} = E^\partial$.
\item $\partial_{\eta} u|_{\partial \Dp \setminus \partial_{\Sph^2} \Dp} = 0$.  
\item $\| u \|_{2, \beta, \gamma; \Dp} + | \mu |   \leq C ( \| E\|_{0, \beta, \gamma-2; \Dp}+
\| E^\partial\|_{1, \beta, \gamma-1; \partial_{\Sph^2} \Dp})$.
\end{enumerate}
\end{prop}
\begin{proof}
We omit the proof, because it is similar to the proof of \cite[Prop. 6.23]{KapWiygul}. 
\end{proof}

\begin{definition}
\label{duv}
Given $(E, E^\partial) \in  C^{0, \beta}_{\sym}(M) \times C^{1, \beta}_{\sym}(\partial_{\Sph^2} M)$, we define
$u_\cat \in C^{2, \beta}_{\sym}(M)$, 
$E_{\D}\in C^{0, \beta}_{\sym}(M)$, 
$E_{\D}^\partial \in C^{1, \beta}_{\sym}(\partial_{\Sph^2} M)$,
$u_\D \in C^{2, \beta}_{\sym}(M)$, 
and
$\mu_E \in \R$
by
\begin{equation}
\label{Evcat}
\begin{gathered}
u_\cat : = \Psicat \kappa^{*-1} v_{\cat}, \quad 
\text{where} \quad 
v_{\cat} : = 
\Rcalhat_{\cat} \left( \kappa^* \Psicat ( \rho^2 E |_{\Sigmacat} ), \kappa^* \Psicat(\rho E^\partial|_{\partial_{\Sph^2} \Sigmacat})\right), \\
E_{\D} : = E -  \Psicat^2 ( E|_{\Sigmacat} ) - [\Lcal_{M}, \Psicat] \kappa^{*-1} v_{\cat},  \\ 
E^{\partial}_{\D} : = E^{\partial} - \Psicat^2 ( E^\partial |_{\partial_{\Sph^2}\Sigmacat} ) - [\partial_{\sigma}, \Psicat] \kappa^{*-1} v_{\cat}|_{\partial_{\Sph^2} M},\\
u_{\D} : = \Psidisk \pi^{*} v_{\D}, 
\quad
\text{where}
\quad
(v_\D, \mu_E) : = \Rcalhat_{\D}( \pi^{*-1} E_\D, \pi^{*-1} E^\partial_\D).
\end{gathered}
\end{equation}
\end{definition}

\begin{definition}
\label{dRappr}
We define a linear map
\begin{align*}
\Rcal^{appr}_M : C^{0, \beta}_{\sym}(M) \times C^{1, \beta}_{\sym}(\partial_{\Sph^2} M)
\rightarrow 
C^{2, \beta}_{\sym}(M) \times \R
\end{align*}
by taking $\Rcal^{appr}_M (E, E^\partial) := (u_\cat+u_\D, \mu_E)$, 
where $u_\cat$, $u_\D$,  and $\mu_E$ were defined in \ref{duv}. 
\end{definition}

\subsection*{The main proposition}
$\phantom{ab}$
\nopagebreak

\begin{prop}[Solvability of the linearized problem on the initial surface]
\label{Plinear}
There is a linear map 
\begin{align*}
\Rcal_{M} : C^{0, \beta}_{\sym}(M) \times C^{1, \beta}_{\sym}(\partial_{\Sph^2} M)
\rightarrow 
C^{2, \beta}_{\sym}(M) \times \R
\end{align*}
such that if $(E, E^\partial)$ belongs to the domain of $\Rcal_{M}$ and $(u, \mu) = \Rcal_{M} (E, E^\partial)$, the following hold.
\begin{enumerate}[label = \emph{(\roman*)} ]
\item $\Lcal_M u = E + \mu(w \circ \pi)$.
\item $\Bcal_{M} u|_{\partial_{\Sph^2} M} = E^{\partial}$.
\item $\partial_\eta u |_{\partial M \setminus \partial_{\Sph^2} M} = 0$. 
\item $\|u \|_{2, \beta, \gamma; M} + |\mu|   \leq C ( \| E \|_{0, \beta, \gamma-2; M}+ 
\| E^\partial \|_{1, \beta, \gamma-1; \partial_{\Sph^2} M})$.
\end{enumerate}
\end{prop}
\begin{proof}
It follows from \eqref{Evcat} and Proposition \ref{pclin} that
\begin{equation}
\label{Euk}
\begin{aligned}
\Lcal_M u_{\cat} &= \Psicat^2 ( E|_{\Sigmacat} ) + [\Lcal_{M}, \Psicat] \kappa^{*-1} v_{\cat}\\
&\phantom{=}+ \Psicat \kappa^{*-1} ( \kappa^* \Lcal_{M} \kappa^{*-1} -\rho^{-2} \Lcalhatcat) v_{\cat} 
\quad \text{and} \\
\partial_{\sigma} u_{\cat}|_{\partial_{\Sph^2} M} 
&= \Psicat^2 ( E^\partial |_{\partial_{\Sph^2}\Sigmacat} ) + [\partial_{\sigma}, \Psicat] \kappa^{*-1} v_{\cat}|_{\partial_{\Sph^2} M}\\
&\phantom{=}+ \Psicat \kappa^{*-1} ( \kappa^* \partial_{\sigma}\kappa^{*-1} +\rho^{-1}(\sgn\vartheta)\partial_{\vartheta}) v_{\cat}|_{\partial_{\Sph^2} M}.
\end{aligned}
\end{equation}

Next, from \eqref{Evcat}, and \ref{dRappr}, it follows that
\begin{equation}
\label{Eud}
\begin{aligned}
\Lcal_{M} u_{\D} &= \Psidisk (E_{\D}|_{\Sigmadisk}) + \mu_E \pi^*w +  [\Lcal_{M}, \Psidisk]\pi^* v_{\D}\\
&\phantom{=}+ \Psidisk \pi^*( \pi^{*-1} \Lcal_{M} \pi^* - \Lcal_{\D}) v_{\D}
\quad \text{and} \\
\partial_{\sigma} u_\D |_{\partial_{\Sph^2} M} 
&= \Psidisk ( E^\partial_{\D}|_{\partial_{\Sph^2} \Sigmadisk}) + 
[\partial_{\sigma}, \Psidisk]\pi^* v_{\D}|_{\partial_{\Sph^2} M}\\
&\phantom{=}+\Psidisk \pi^* ( \pi^{*-1} \partial_{\sigma}\pi^* +\partial_{\rr}) v_{\D}|_{\partial_{\Sph^2} M}
- u_{\D}|_{\partial_{\Sph^2} M}.
\end{aligned}
\end{equation}

We now define $L : C^{2, \beta}_{\sym}(M) \times \R \rightarrow C^{0, \beta}_{\sym}\times C^{1, \beta}_{\sym}(\partial_{\Sph^2} M)$ by
\begin{align*}
L(u, \mu) = (\Lcal_M u - \mu w \circ \pi, \Bcal_{M} u|_{\partial_{\Sph^2} M}).
\end{align*}
It follows from \ref{dloc}, \eqref{Evcat}, \eqref{Euk}, and \eqref{Eud} that
\begin{align}
L \Rcal^{appr}_M( E, E^\partial) = (\Etilde, \Etilde^\partial),
\end{align}
where
\begin{equation}
\label{Elinerr}
\begin{aligned}
\Etilde &= E + \Psicat \kappa^{*-1}( \kappa^* \Lcal_M \kappa^{*-1} - \rho^{-2}\Lcalhatcat) v_{\cat}\\
&\phantom{=}+
\Psidisk \pi^*(\pi^{*-1} \Lcal_M \pi^* - \Lcal_{\D}) v_{\D} + [\Lcal_M, \Psidisk]\pi^* v_{\D},
\end{aligned}
\end{equation}
and
\begin{equation}
\label{Elinerrb}
\begin{aligned}
\Etilde^\partial &= E^\partial + 
\Psicat \kappa^{*-1} ( \kappa^* \partial_{\sigma}\kappa^{*-1} +\rho^{-1}(\sgn\vartheta)\partial_{\vartheta}) v_{\cat}|_{\partial_{\Sph^2} M}\\
&\phantom{=} + \Psidisk \pi^* ( \pi^{*-1} \partial_{\sigma}\pi^* + \partial_{\rr}) v_\D|_{\partial_{\Sph^2} M}
+ [\partial_{\sigma}, \Psidisk]\pi^* v_\D|_{\partial_{\Sph^2} M} + u_{\cat}|_{\partial_{\Sph^2} M}.
\end{aligned}
\end{equation}
Using \ref{Lopdiff}, \ref{pclin}, and \ref{pdlin} to estimate \eqref{Elinerr} and \eqref{Elinerrb}, we obtain that
\begin{multline*}
\| \Etilde - E\|_{0, \beta, \gamma-2; M} 
+
\| \Etilde^\partial - E^\partial \|_{1, \beta, \gamma-1; \partial_{\Sph^2} M}
\leq\\
\leq
C
( \tau^{2\alpha}+ \tau^{2(1-\alpha)}|\log \tau|^2 + \tau^\alpha)
\left(
\| E\|_{0, \beta, \gamma-2; M}
+
\| E^\partial \|_{1, \beta, \gamma-1; \partial_{\Sph^2} M}
\right).
\end{multline*}
Thus $L \Rcal^{appr}_M$ is a small perturbation of the identity operator on $C^{0,\beta}_{\sym}(M) \times C^{1, \beta}_{\sym}(\partial_{\Sph^2} M)$ and hence is invertible.  Setting
\begin{align*}
\Rcal_M : = \Rcal^{appr}_M(L \Rcal^{appr}_M)^{-1}
\end{align*}
concludes the proof of (i)-(iii), and the estimates in (iv) follow from the estimates completed above. 
\end{proof}

\section[The main theorem]{The main theorem}
\label{S:main}

\subsection*{The nonlinear terms} 
$\phantom{ab}$
\nopagebreak

 Using rescaling, we prove the following global estimate for the nonlinear terms of the mean curvature of $M_\phi$. 

\begin{lemma}
\label{Lquad}
If $M$ is an initial surface as in \ref{dinit} and $\phi \in C^{2, \beta}(M)$ satisfies $\| \phi \|_{2, \beta, \gamma; M} \leq \tauunder^{1+ \alpha/4}$ (recall \ref{dnorm}), then $M_\phi$ is a well-defined embedded surface.  Moreover, if $H_\phi$ denotes the mean curvature of $M_\phi$ pulled back to $M$ by $X^{\R^3, \gaux}_{M, \phi \nu}$, and $H$ is the mean curvature of $M$, then
\begin{align}
\label{Equad}
\| H_\phi - H - \Lcal_M \phi \|_{0,\beta, \gamma-2; M} 
\leq 
C 
\tauunder^{\gamma - 1}\| \phi \|^2_{2, \beta, \gamma; M}.
\end{align}
\end{lemma}
\begin{proof}
Fix any $q \in M$, and write $B$ for the geodesic ball in $(M, \rho(q)^{-2} g)$ with center $q$ and radius 1.  Clearly for each nonnegative integer $k$ there is a constant $C(k)$ such that the first $k$ $\rho^{-2}(q)g$ covariant derivatives of $\rho^{-2}(q) \gaux$ are, as measured by $\rho^{-2}(q) g$, bounded by $C(k)$.  Moreover, by \ref{Pinit}, we can choose $C(k)$ so that it also bounds the first $k$ covariant derivatives of the second fundamental form in $(\R^3, g)$ of $\rho^{-1}(q) X|_B$, the inclusion map of $B$ blown up by a factor of $\rho^{-1}(q)$.

By the definition of the norm and since $\| \phi \|_{2, \beta, \gamma; M} \leq \tauunder^{1+\alpha/4}$, we have that the restriction of $\phi$ on $B$ satisfies
\begin{align*}
\| \rho^{-1}(q) \phi : C^{2, \beta}(B, \rho^{-2}(q) g)\| 
\leq
C \rho^{\gamma-1}(q) \| \phi \|_{2, \beta, \gamma; M}.
\end{align*}
Since the right hand side is small in absolute terms, we conclude that $X^{\R^3, \gaux}_{M, \phi \nu}$ is well-defined on $B$, its restriction to $B$ is an embedding, and using scaling for the left hand side that
\begin{align*}
\| \rho(q) (H_\phi - H - \Lcal_M \phi) : C^{0, \beta}( B, \rho^{-2}(q)g )\| 
\leq 
C \rho^{2\gamma - 2}(q) \| \phi \|^2_{2, \beta, \gamma; M}.
\end{align*}
We conclude that
\begin{align*}
\rho^{2- \gamma}(q) \| (H_\phi - H - \Lcal_M \phi) : C^{0, \beta}( B, \rho^{-2}(q)g )\| 
\leq
C \rho^{\gamma -1}(q) \| \phi\|^2_{2, \beta, \gamma; M}.
\end{align*}
From this, the estimate in the statement of the lemma follows by using the definitions.
\end{proof}

\subsection*{The fixed point theorem}
$\phantom{ab}$
\nopagebreak 

\begin{lemma}[Diffeomorphisms $\Fzeta$]
\label{Ldiffeo}
There exists a family of $\groupdisk$-equivariant diffeomorphisms $\Fzeta : M[0] \rightarrow M[\zeta]$ which depend continuously on $\zeta$ and have the property that for any $u \in C^{2, \beta}_{\sym}(M[\zeta])$, 
\begin{align}
\label{Enormcomp}
\| u \|_{2, \beta, \gamma; M[\zeta]} \Sim_{2} \| u \circ \Fzeta\|_{2, \beta, \gamma; M[0]}. 
\end{align}
\end{lemma}
\begin{proof}
We first define a family of diffeomorphisms $\Fzetacir : \Sigmacir[0] \rightarrow \Sigmacir[\zeta]$ identifying the pre-initial surfaces.  We set $\aunder : = a[0]$, the value of the quantity $a$ defined in \eqref{Ea}, when $\zeta =0$.  The map $T[\zeta] : \cyl^+_{[0, \aunder]} \rightarrow \cyl^+_{[0, a]}$ defined by
\begin{align*}
T[\zeta] (t, \vartheta)  = (t a/\aunder, \vartheta)
\end{align*}
is clearly a diffeomorphism. 
On the other hand, if we set
\begin{align*}
A[\zeta] &: = \Sigmacir \setminus ( \kappacir[\zeta]( \cyl^+_{[0, a-2]})) \subset \Sigmacird, \\
B[\zeta] &: = \Sigmacir \setminus ( \kappacir[\zeta]( \cyl^+_{[0, a-1]})) \subset A[\zeta], \\
C[\zeta] &: = \Sigmacircat \cap ( A[\zeta] \setminus B[\zeta]) \subset \Sigmacircat\cap \Sigmacird,
\end{align*}
then, provided $\varpi$ is small enough, 
\begin{align*}
\picir [ \zeta](A) \subset \picir[0] ( \Sigmacird[0])
\quad \text{and} 
\quad
\picir [ \zeta](C) \subset \picir[0] (\Sigmacircat[0] \cap \Sigmacird),
\end{align*}
so that the maps $\Fcalcir'[\zeta]: A \rightarrow \Sigmacird[0]$ and 
$T'[\zeta] :  \cyl^+_{[a-2, a-1]} \rightarrow \cyl^+_{[0, \aunder]}$
defined by
\begin{align*}
\Fcalcir'[\zeta] : = \picir[0]^{-1} \circ \picir[\zeta]
\quad 
\text{and}
\quad
T'[\zeta]:= 
\kappacir[0]^{-1}[\zeta] \circ \Fcalcir'[\zeta] \circ \kappacir[\zeta]
\end{align*}
are well-defined and diffeomorphisms onto their images.
We now glue together the maps $T$ and $T'$ to obtain a smooth diffeomorphism $\Ttilde [\zeta] : \cyl^+_{[0, a-1]} \rightarrow \cyl^+_{[0, \aunder]}$ by defining
\begin{align*}
\Ttilde[\zeta] : = \Psibold[a-2, a-1; t ]( T, T').
\end{align*}
Finally, we define $\Fzetacir : \Mcir[\zeta] \rightarrow \Mcir[0]$ to be the unique map having the restrictions
\begin{align*}
\Fzetacir |_{B[\zeta]} : = \Fcalcir'[\zeta] 
\quad \text{and}
\quad
\Fzetacir|_{\kappacir[\zeta](\cyl^+_{[0, a-1]})} := \kappacir[\zeta]\circ \Ttilde \circ \kappacir[\zeta]^{-1}.
\end{align*}
This concludes the definition of $\Pcir[\zeta]$.  We then define
\begin{align*}
\Fzeta : = \Xcir_{\ucir[\zeta]}[\zeta] \circ \Fzetacir \circ \Xcir_{\ucir[0]}[0]^{-1},
\end{align*}
where $\Xcir_{\ucir} : \Mcir \rightarrow M$ was defined in \ref{dinit}. 

Using \eqref{Etau} and \eqref{Ea}, it is easy to check that
\begin{align}
\aunder \Sim_{1+C/|\log \varpi|^2} a.
\end{align}
Using this, it is not difficult to see that \eqref{Enormcomp} holds by taking $\varpi$ small enough.
\end{proof}

\begin{theorem}[Main theorem]
\label{Tmain}
There exists $\varpi_0>0$ such that for any $\varpi \in (0, \varpi_0]$,  there exist
\begin{align*}
 \zetabreve \in [-\tauunder, \tauunder]
 \quad
 \text{and}
 \quad
 \phibreve \in C^{\infty}_{\sym}(M[\zetabreve])
 \quad
 \text{satisfying}
 \quad
 \| \phibreve \|_{2,\beta, \gamma; M[\zetabreve]} \leq \tauunder^{1+\alpha/4},
 \end{align*}
 such that the perturbation $\Mbreve : = ( M [ \zetabreve])_{\phibreve}$ satisfies the following properties. 
\begin{enumerate}[label = \emph{(\roman*)} ]
\item $\Mbreve \subset \B^3$  is an embedded $\groupdisk$-invariant minimal disk.
\item $\partial \Mbreve \subset  \partial \B^3 \cup P\cup \Pprime$.
\item $\Mbreve \cup \Rcapunder_P \Mbreve$ and $\Mbreve \cup \Rcapunder_{\Pprime} \Mbreve$ are smooth disks with corners. 
\item $\Mbreve$ meets $\partial \B^3$ orthogonally along $\partial_{\Sph^2} \Mbreve$.
\end{enumerate}
\end{theorem}
\begin{proof}
The proof is based on finding a fixed point for a map $\Jcal$ we will define shortly.  Define 
\begin{align}
\label{EB}
B : = \left\{ u \in C^{2, \beta}_{\sym}(M[0]) : \| u \|_{2, \beta, \gamma; M} \leq \tauunder^{1+\alpha}\right\} \times [- \tauunder, \tauunder],
\end{align}
and suppose $(u, \zeta) \in B$ is given.  Use \ref{Plinear} to define $(u_H, \mu_H): = -\Rcal_{M[\zeta]}( H -  (\Mcal \varphi) w \circ \pi, 0)$. Define also $\phi \in C^{2, \beta}_{\sym}(M[\zeta])$ by $\phi = u \circ \Fzeta^{-1} + u_H$.  We then have:
\begin{enumerate}
\item $\Lcal_M u_H +H=  ( \Mcal \varphi + \mu_H) w\circ \pi$. 
\item By \ref{LglobalH}, \ref{Plinear}, \ref{Ldiffeo}, and the size of $u$ in \eqref{EB}, 
	$|\mu_H|+ \| \phi \|_{2, \beta, \gamma; M[\zeta]} \leq \tauunder^{1+\alpha/4}.$
\item Using \ref{Plinear} again, we define $(u_Q, \mu_Q) : = - \Rcal_{M[\zeta]}(H_\phi - H - \Lcal_M \phi, 0)$.  We then have
	\begin{equation}
	\label{EHtilde}
	\begin{gathered}
		\Lcal_M (u_Q - u \circ \Fzeta^{-1}) + H_{\phi} = ( \Mcal \varphi + \mu_H + \mu_Q) w \circ \pi, \\
	\Bcal_M \phi|_{\partial_{\Sph^2} M} = 0.
	\end{gathered}
	\end{equation}
\item Moreover by (2), \ref{Plinear}, and \ref{Lquad}, $\| u_Q\|_{2, \beta, \gamma; M[\zeta]}+ |\mu_Q| \leq C \tauunder^{(3+\alpha)/2}$. 
\end{enumerate}

We then define $\Jcal : B \rightarrow C^{2, \beta}_{\sym} (M[0]) \times \R$ by 
\begin{align}
\label{EJcal}
\Jcal ( u, \zeta) = \left(u_Q \circ \Fzeta, \zeta + \Mcal \varphi+ \mu_H + \mu_Q\right).
\end{align}

We are now ready for the fixed-point argument. 
By using \ref{Lpmis}, \ref{Ldiffeo}, and (2) and (4) above, it is straightforward to check that $\Jcal(B) \subset B$.  Clearly $B$ is convex.  Let $\beta' \in (0, \beta)$.  The inclusion $B \hookrightarrow C^{2, \beta'}_{\sym}(M[0])\times \R$ is compact by the Ascoli-Arzela theorem.  By inspecting the proof of \ref{Plinear}, it is easy to see that $\Rcal_{M[\zeta]}$ depends continuously on $\zeta$.  From this and the definitions, it is easy to check that $\Jcal$ is a continuous map in the induced topology.  By Schauder's fixed point theorem \cite{gilbarg}, there is a fixed point $(\breve{u}, \breve{\zeta})$ of $\Jcal$.  Using \eqref{EJcal} and the fixed point property, we see that $u_Q = u \circ \Fzeta^{-1}$ and $\Mcal \breve{\varphi} + \breve{\mu}_H + \breve{\mu}_Q = 0$, where we use ``$\breve{\phantom{a}}$'' to denote the various quantities for $(v, \zeta) = (\breve{v}, \breve{\zeta})$.  

By \eqref{EHtilde} we conclude the minimality of $\breve{M}$.  The embeddedness of $\Mbreve$ follows from \ref{Lquad}, and the symmetries of $\Mbreve$ in (i) as well as items (ii)-(iii) follow from combining \ref{LMpertang}, \ref{Plinear}, and \ref{Lpm} with the definition of $\phi$ above.  Finally, (iv) follows from combining the boundary condition in \ref{EHtilde} with \ref{LMpertang}(iii). 
\end{proof}

\begin{corollary}[Immersed free boundary annuli]
\label{Cmain}
Fix $\varpi \in (0, \varpi_0]$ and $\Mbreve$ as in \ref{Tmain}.  Then
\begin{align*}
\Mhat = \grouptheta \Mbreve
\end{align*}
is a smooth free boundary minimal surface immersed in $\B^3$.  Moreover, if $\varpi$ is a rational multiple of $\pi$, then $\Mhat$ is a properly immersed compact annulus.  
On the other hand, if $\varpi/ \pi$ is irrational, then $\Mhat$ is a noncompact and nonproperly immersed strip. 
\end{corollary}
\begin{proof}
The smoothness of $\Mhat$ follows immediately from the smoothness of $\Mbreve$ and \ref{Tmain}(iii).  By construction $\partial \Mhat \subset \partial \B^3$, and by combining this with \ref{Tmain}(i) and (iv) it follows that $\Mhat$ is a free boundary minimal surface.  Finally the remaining assertions are clear, since $\grouptheta$ is a finite group if and only if $\varpi$ is a rational multiple of $\pi$. 
\end{proof}


\appendices
\section*{Appendices}

\section{Graphs and the auxiliary metric}
\label{Sgraphs2}
\nopagebreak

\begin{definition}[Vector fields and sliding]
\label{Dvslide}
Given an open set $U \subset P$ and a vector field $V$ defined on a domain containing $U$, we define $\DDD_V : U \rightarrow P$ by $\DDD_V : = \exp^{P, g} \circ V|_{U}$.  We also define $\Utilde_V : = U \cap \DDD_V(U)$.
\end{definition}

\begin{lemma}
\label{LVmvt}
If $U, V$, and $\DDD_V$ are as in \ref{Dvslide} and $f \in C^\infty(\Utilde_V)$, then 
\begin{align*}
	\| f \circ \DDD_V - f : C^k(\Utilde_V)\|  \leq
	C(k) \| f: C^{k+1}(U, g)\| \| V : C^k( U, g)\|.
\end{align*}
If $\DDD_V$ is moreover a diffeomorphism and $\| V : C^k(U, g)\|$ is small enough, then additionally
\begin{align*}
	\| f \circ \DDD^{-1}_V - f : C^k(\Utilde_V)\|  \leq
	C(k) \| f: C^{k+1}(U, g)\| \| V : C^k( U, g)\|.
\end{align*}
\end{lemma}
\begin{proof}
This is a consequence of the mean value theorem and a straightforward induction argument.
\end{proof}

\begin{assumption}
\label{Auaux}
We now assume given the following:
\begin{enumerate}[label=(\roman*)]
\item A domain $U \subset P \setminus D_0(2/3)$.
\item A function $u \in C^\infty(U)$ with $\| u : C^k(U, g)\|$ as small as needed in absolute terms.
\end{enumerate}
\end{assumption}

\begin{lemma}
\label{Luauxw}
Given $u$ as in \ref{Auaux}, there is a vector field $V$ on $U$ uniquely determined by $\DDD_V = \Pi^{\gaux}_P \circ  X^{\R^3,g}_{U, u \nu}$ and a function $w : \DDD_V(U) \rightarrow \R$ uniquely determined by $X_{U, u\nu}^{\R^3, g} = X^{\R^3,\gaux}_{\DDD_V(U), w \nu} \circ \DDD_V$.  In other words, the diagram
 \begin{equation}
\label{ediag3}
\begin{tikzcd}
\R^3 \arrow[r, rightarrow, "\id"] 
               & \R^3 \arrow[d, rightarrow, "\Pi^{\gaux}_P"']\\
U \arrow[r, rightarrow, "{\DDD_V}"] \arrow[u, rightarrow, "X^{\R^3, g}_{U, u \nu }"] & \DDD_V(U) \arrow[u, bend right=25, pos=.45,  "X^{\R^3, \gaux}_{\DDD_V(U), w \nu}"']
\end{tikzcd}
\end{equation} 
commutes. Moreover, the following hold.
\begin{enumerate}[label=\emph{(\roman*)}]
\item $\| V : C^k(U, g) \| \leq C(k) \| u : C^k(U, g)\|^2.$
\item $w =\Omega \cdot  \sin^{-1} \circ ( \frac{1}{\Omega} \cdot u \circ \DDD^{-1}_V)$.
\item $\| w - u : C^k(\Utilde_V, g) \| \leq C(k) \| u : C^{k+1}(\Utilde_V, g)\| \| u: C^{k}(\Utilde, g)\|^2$.
\end{enumerate}
\end{lemma}
\begin{proof}
By the definitions, we have $X^{\R^3, g}_{U, u\nu }(p) = p + u(p) \nu(p)$.  On the other hand, for any $p \in \Omega$ and any $z \in \R$, it follows from \ref{dgaux} and \ref{Auaux}(i) that
\begin{align}
\label{Egauxexp}
\exp^{\R^3, \gaux}_p(z \nu(p)) 
= 
\cos \left( \frac{z}{|p|}\right) p + |p| \sin \left(\frac{z}{|p|}\right) \nu(p).
\end{align}
By combining these observations, it follows that
\begin{align*}
(\Pi^{\gaux}_P \circ X^{\R^3, g}_{\Omega, u\nu})(p)
=
p \sqrt{1+ (u(p)/|p|)^2},
\end{align*}
and therefore that requesting the equation $\DDD_V = \Pi^{\gaux}_P \circ  X^{\R^3,g}_{U, u \nu}$ uniquely determines $V$ by
\begin{align}
V(p) = p( \sqrt{1+ (u(p)/|p|)^2} - 1).
\end{align}
Item (i) follows from this.

Next, it follows from \eqref{Egauxexp}, the smallness assumption on $u$ in \ref{Auaux}(ii), and the definitions that the commutativity of \eqref{ediag3} is equivalent to (ii).

Finally, we have by estimating using (ii) that 
\begin{align*}
\| w - u \circ \DDD^{-1}_V : C^k(\Utilde_V, g)\| \leq C(k) \| u \circ \DDD^{-1}_V : C^k(\Utilde_V, g) \|^3. 
\end{align*}
On the other hand, by combining \ref{LVmvt} and (i) we have that
\begin{align*}
\| u \circ \DDD^{-1}_V - u : C^k(\Utilde_V, g) \| \leq
C(k) \| u : C^{k+1}(U, g)\| \| u : C^k(U, g) \|^2.
\end{align*}
Combining the preceding completes the proof of (iii).
\end{proof}

\bibliographystyle{abbrv}
\bibliography{bibliography}

\end{document}